\documentclass[a4paper,leqno,twoside,11pt]{amsart}
\usepackage[T1]{fontenc}
\usepackage{amsmath,amsthm,amsfonts,amssymb}

\usepackage{fixltx2e}
\usepackage{ifpdf}
 \ifpdf
  \usepackage[pdftex]{graphicx}
  \usepackage[protrusion=true,expansion=true]{microtype}
 \else
  \usepackage[dvips]{graphicx}
 \fi

\usepackage[margin=3.5cm]{geometry}

\newtheorem{theorem}{Theorem}[section]
\newtheorem{lemma}[theorem]{Lemma}
\newtheorem{proposition}[theorem]{Proposition}
\newtheorem{corollary}[theorem]{Corollary}

\theoremstyle{remark}
\newtheorem{observation}{Observation}

\theoremstyle{definition}
\newtheorem{definition}[theorem]{Definition}
\newtheorem{remark}[theorem]{Remark}

\newtheorem*{lemma*}{Lemma}
\newtheorem*{theorem*}{Theorem}
\newtheorem*{proposition*}{Proposition}
\newtheorem*{corollary*}{Corollary}
\newtheorem*{definition*}{Definition}

\newcommand{\C}{\mathbb{C}}
\newcommand{\N}{\mathbb{N}}

\newcommand{\R}{\mathbb{R}}

\newcommand{\Z}{\mathbb{Z}}
\newcommand{\n}{\mathfrak{n}}
\newcommand{\brac}[2]{\langle#1,#2\rangle}
\newcommand{\z}{\mathfrak{z}}
\renewcommand{\v}{\mathfrak{v}}
\renewcommand{\a}{\mathfrak{a}}
\newcommand{\s}{\mathfrak{s}}

\newcommand{\romannum}{\renewcommand{\labelenumi}{\textnormal{(\roman{enumi})}}}


\newcommand{\refr}[1]{(\ref{#1})}

\usepackage{tabularx}
\newcommand{\cfct}{\mathbf{c}}
\newcommand{\dfct}{\mathbf{d}}

\title[Jacobi transform multipliers and fractional integration]{$L^p$-results for fractional integration and multipliers for the Jacobi transform}
\author{Troels Roussau Johansen}
\address{Mathematisches Seminar,\\
   Christian-Albrechts Universit\"at zu Kiel\\
   Ludewig-Meyn-Strasse 4, D-24098 Kiel\\
   Germany}
\email{johansen@math.uni-kiel.de}

\begin{document}

\begin{abstract}
We use precise asymptotic expansions for Jacobi functions
$\varphi^{(\alpha,\beta)}_\lambda$ parameters $\alpha$, $\beta$
satisfying $\alpha>\frac{1}{2}$, $\alpha>\beta>-\frac{1}{2}$, to generalizing classical H\"ormander-type multiplier theorem
for the spherical transform on a rank one Riemannian symmetric space (by Clerc/Stein and Stanton/Tomas) to the
framework of Jacobi analysis. In particular, multiplier results for the
spherical transform on Damek--Ricci spaces are subsumed by this approach, and it yields multiplier results for the hypergeometric `Heckman--Opdam transform' associated with a rank one root system.
We obtain near-optimal $L^p-L^q$ estimates for the integral operator associated with the convolution kernel $m_a:\lambda\mapsto(\lambda^2+\rho^2)^{-a/2}$, $a>0$.
\end{abstract}
\maketitle
\addtocounter{section}{-1}
\section{Disclaimer}
The present preprint will not be submitted for publication since the main result -- the Hörmander-Mihlin multiplier theorem for the Jacobi transform -- is a special case of results from Bloom/Xu: \emph{Fourier multipliers for {$L^p$} on {C}h\'ebli-{T}rim\`eche hypergroups}, Proc. London Math. Soc. (3) \textbf{80} (2000), 643-664.

However, the proof is different (it uses transference for hypergroups, appearing in \cite{Gigante}, as well as a precise asymptotic expansion of Jacobi functions), so it directly extends the methods from \cite{Stanton-Tomas}. There are several technical issues to overcome in this extension to Jacobi analysis, and since we use the same techniques elsewhere, we have decided to post the results in order to have a convenient reference. 

\section{Introduction}
Spherical functions on noncompact Riemannian symmetric spaces of rank one
behave like classical Bessel functions close to the origin but behave quite
differently at infinity, as was made precise in \cite{Stanton-Tomas}. It was
realized early on that these spherical functions coincided with certain special
functions $\varphi^{(\alpha,\beta)}_\lambda$ -- called Jacobi functions -- for
suitable positive integer parameters $\alpha,\beta$. The analogous asymptotic
behavior of $\varphi^{(\alpha,\alpha)}_\lambda$ was investigated at length in
\cite{Schindler}. The present paper is also about the detailed asymptotic
behavior of the Jacobi functions, now for arbitrary complex parameters
satisfying the condition that $\Re\alpha>\frac{1}{2}$, $\Re\alpha>\Re\beta>-\frac{1}{2}$; we always write $\Re z$ for the real part and $\Im z$ for the imaginary part of a complex number $z$.

One of the  earliest studies of multipliers in the context of noncompact
Riemannian symmetric spaces is the important paper \cite{ClercStein}, where
several key ideas towards a proof of multiplier theorems to come were first
presented. The authors certainly knew, although it was not thusly designated at
the time, that one ought to be able to `transfer' multipliers for the spherical
transform on $G/K$ to a multiplier for the Euclidean Fourier transform on the
Euclidean space $\mathfrak{a}$ coming from the $KAK$-decomposition of $G$. Once
in Euclidean space, one should apply the H\"ormander--Mikhlin multiplier
theorem and then transfer the results back to $G/K$. This approach was carried
out in detail in \cite{Stanton-Tomas}, where the relevant transference
principle is stated as well. A different approach was taken by Anker in
\cite{Anker-Annals}, where somewhat more precise multiplier results are derived
for higher rank symmetric spaces. His approach is based on a detailed study of
the Abel transform, and while there \emph{is} an Abel transform in the Jacobi
setting as well, it is not nearly as well-behaved, and we have decided to
forego it in the present paper. We thus obtain multiplier results for the
Jacobi transform analogous to the results contained in
\cite[Section~5]{Stanton-Tomas} for the spherical Fourier transform.
Summability and almost everywhere convergence results for the Jacobi transform, analogous to
\cite{Meaney-Prestini}, has been treated in \cite{Johansen-disc}.

In Section \ref{section.transference}, we explain such a principle in the
Jacobi setting. A principle of transference, in the sense of Coifman and Weiss,
is a statement that relates norm estimates for convolution operators acting on
different spaces. The present paper was written in part to
understand to what extend this classical method could be generalized, but since
convolution in the more general Jacobi setting is not nearly as pleasant,
experts in the seventies perhaps did not consider the natural extension. As
will become apparent below, the missing link was the observation that the
Jacobi convolution could be described in terms of a structure known as a
\emph{hypergroup}. Once we realized that this could be done, it was a trivial
matter to discover the paper \cite{Gigante}, where a transference result is
given for real parameters $\alpha$, $\beta$ satisfying $\alpha\geq\beta\geq
-\frac{1}{2}$.

In Section \ref{section.multipliers} we introduce the natural notion of a
multiplier for the Jacobi transform, show that they necessarily must extend holomorphically into a suitable strip in $\C$,
and we formulate one of the two multiplier
theorems in Theorem \ref{thm.hormander.multiplier}. The proof will occupy more
than half the paper, since we have to redo some of the lengthy proofs in
\cite{Stanton-Tomas} for real parameters $\alpha$, $\beta$ and complex
spectral parameter $\lambda$. One benefit is that we are able to cover, in a uniform way,
all rank one symmetric spaces, all Damek--Ricci spaces, as
well as type $BC$ root systems. We shall
explain these examples at length in Section \ref{section.examples}. At this
point we should also like to mention that the results from Section
\ref{section.local} and Section \ref{section.long-range} will be utilized in
companion papers: In \cite{Johansen-disc} we investigate almost everywhere convergence of
the inverse Jacobi transform and properties of a disc multiplier, and in \cite{Johansen-exp2} we extend a result due to Giulini, Mauceri, and Meda, to the effect that certain non-integrable multipliers are allowed. Applications to almost everywhere convergence of Bochner--Riesz means will be given elsewhere.

In Section \ref{sec.fract} we examine $L^p-L^q$ mapping properties of Riesz transforms (``fractional integration''). The study of such transforms and more general potentials of the form $(zI-\Delta)^{-s}$ have attracted much attention over the years, and continue to be relevant. We certainly cannot adequately account for the vast literature on the topic, so we simply acknowledge the papers that motivated the present extension to Jacobi functions:  As is already apparent, our point of departure was the fundamental paper \cite{Stanton-Tomas}, where the results that we generalize appeared in Section 6. Using methods from spectral geometry of Riemannian manifolds, together with precise kernel estimates, the study of $L^p-L^q$ mapping properties were later carried out both for higher rank symmetric spaces (\cite{Anker-Lohoue}, \cite{Anker.duke},\cite{Cowling_Giulini_Meda1}) and for more general Riemannian manifolds (\cite{Lohoue},\cite{Taylor.duke}), to name a few. Our method is more elementary and does not rely on any geometric properties (except perhaps noncompactness) of the underlying space (which is just $\R$) but instead rests heavily on precise information on objects like the $\mathbf{c}$-function and asymptotic behavior of spherical functions. Such was also the approach in \cite{Stanton-Tomas}, the advantage being that we can immediately cover nonsymmetric spaces (Damek--Ricci spaces), and even beyond.  We have added some details in the interpolation arguments needed for the main result, Theorem \ref{thm.frac}, as well as fixing a small gap in an argument pertaining to Lorentz space estimates, so that the proof is lengthier than the symmetric space analogue, \cite[Theorem~6.1]{Stanton-Tomas}.
\medskip

A final remark on the choice of parameters $\alpha,\beta$ must be made: While deriving the asymptotic expansions for $\varphi_\lambda^{(\alpha,\beta)}$ is uneventful even for complex parameters, serious problems crop up when we study the relevant convolution inequalities. Since we would now have to integrate against a complex measure, and since the convolution `kernel' would not give rise to a probability density anymore, standard inequalities like the H\"older and Hausdorff--Young inequalities would therefore have to be rewritten. Other statements simply do dot hold anymore. We are grateful to Margit R\"osler for pointing out these issues.

\section{Preliminary Remarks on Jacobi Functions}
We briefly recall some pertinent facts on Jacobi functions. A much more
detailed account may be found in \cite{Koornwinder-book}. Let $(a)_0=1$ and
$(a)_k=a(a+1)\cdots (a+k-1)$. The hypergeometric function ${_2}F_1(a,b;c,z)$ is
defined by
\[{_2}F_1(a,b;c,z)=\sum_{k=0}^\infty\frac{(a)_k(b)_k}{(c)_kk!}z^k,\quad\vert z\vert<1;\]
the function $z\mapsto{_2}F_1(a,b;c,z)$ is the unique solution of the differential equation
\[z(1-z)u''(z)+(c-(a+b+1)z)u'(z)-abu(z)=0\]
which is regular in $0$ and equals $1$ there. The Jacobi function with
parameters $(\alpha,\beta)$ is defined by
$\varphi_\lambda^{(\alpha,\beta)}(t)={_2}F_1(\frac{1}{2}(\alpha+\beta+1-i\lambda),
\frac{1}{2}(\alpha+\beta+1+i\lambda); \alpha+1,-\sinh^2t)$. For
$\vert\beta\vert<\alpha+1$, the system
$\{\varphi_\lambda^{(\alpha,\beta)}\}_{\lambda\geq 0}$ is a continuous
orthonormal system in $\R^+$ with respect to the weight
$\Delta_{\alpha,\beta}(t)=(2\sinh t)^{2\alpha+1}(2\cosh t)^{2\beta+1}$, $t>0$.

In what follows we assume that $\alpha\neq -1,-2,\ldots$, $\alpha>\frac{1}{2}$, and $\alpha>\beta>-\frac{1}{2}$. We shall consider the more general case of complex parameters satisfying the relations $\Re\alpha>\frac{1}{2}$,  $\Re\alpha>\Re\beta+1>$ and $\Re\beta>-\frac{1}{2}$ in Appendix \ref{app.asymptotic}, but as far as serious analysis with convolution inequalities and $L^p$-spaces is concerned, we need this restriction. The added requirement $\alpha>\frac{1}{2}$ is unnecessary as far as the general Jacobi analysis goes but is needed for the asymptotic analysis. The usual Lebesgue space on $\R^+$ shall simply be denoted $L^p$, whereas by $L^p(d\mu)$ we understand the weighted Lebesgue space, with $d\mu(t)=d\mu_{\alpha,\beta}(t)=\Delta(t)\,dt$. Let $\rho=\alpha+\beta+1$. We adopt the notational convention of writing $\mu(A)$ for the weighted measure of a measurable subset $A$ of $\R$, that is, $\mu(A)=\|1_A\|_{L^1(d\mu)}$. It is of paramount importance to stress that the behavior of $\Delta(t)$ depends on the `size' of $t$. More precisely, $\vert\Delta_{\alpha,\beta}(t)\vert \leq t^{2\alpha+1}$ for $t\lesssim 1$, whereas $\vert\Delta_{\alpha,\beta}(t) \lesssim e^{2\rho t}$ for $t\gg 1$

The Jacobi-Laplacian is the operator
$\mathcal{L}=\mathcal{L}_{\alpha,\beta}=\frac{d^2}{dt^2}+((2\alpha+1)\coth
t+(2\beta+1)\tanh t)\frac{d}{dt}$, by means of which the Jacobi function
$\varphi_\lambda^{(\alpha,\beta)}$ may alternatively be characterized as the
unique solution to
\begin{equation}\label{eqn.eigeneqn}
\mathcal{L}_{\alpha,\beta}\varphi+(\lambda^2+\rho^2)\varphi=0
\end{equation}
on $\R^+$ satisfying $\varphi_\lambda(0)=1$ and $\varphi_\lambda'(0)=0$. It is
thereby clear that $\lambda\mapsto\varphi_\lambda(t)$ is analytic for all
$t\geq 0$. Moreover, for $\Im\lambda\geq 0$, there exists a unique
solution $\phi_\lambda$ to the same equation satisfying
$\phi_\lambda(t)=e^{(i\lambda-\rho)t}(1+o(1))$ as $t\to\infty$, and
$\lambda\mapsto\phi_\lambda(t)$ is therefore also analytic for $t\geq 0$.

In analogy with the case of symmetric spaces, one proceeds to show the existence of a function
$\cfct=\cfct_{\alpha,\beta}$ for which
$\varphi_\lambda(t)=\cfct(\lambda)e^{(i\lambda-\rho)t}\phi_\lambda(t)+\cfct(-\lambda)e^{(-i\lambda-\rho)t}\phi_{-\lambda}(t)$.
Since we adhere to the conventions and normalization used in
\cite{Koornwinder-FJ}, the $\cfct$-function is given by
\[\cfct(\lambda)=\frac{2^\rho\Gamma(i\lambda)\Gamma(\frac{1}{2}(1+i\lambda))}
{\Gamma(\frac{1}{2}(\rho+i\lambda))\Gamma(\frac{1}{2}(\rho+i\lambda)-\beta)}.\]
Observe that for $\alpha,\beta\neq -1,-2,\ldots$, $\cfct(-\lambda)^{-1}$ has
finitely many poles for $\Im\lambda <0$ and none if $\Im\lambda\geq 0$ and $\Re\rho>0$. It follows from Sterling's formula that
for every $r>0$ there exists a positive constant $c_r$ such  that
\begin{equation}\label{eqn.c-fct.est}
\vert\cfct(-\lambda)\vert^{-1}\leq c_r(1+\vert\lambda\vert)^{\Re\alpha+\frac{1}{2}}\text{ if } \Im\lambda\geq 0\text{ and }
\cfct(-\lambda')\neq 0\text{ for }\vert\lambda'-\lambda\vert\leq r;
\end{equation}
for easy reference we also recall from \cite[Lemma~2.1]{Johansen-disc} the following estimates.

\begin{lemma}\label{lemma.precise-c} Assume $\alpha>\beta>-\frac{1}{2}$.
 \begin{enumerate}
 \romannum
 \item For every integer $M$ there exist constants $c_i, i=0,\ldots,M-1$ (depending on $\alpha$, $\beta$, and $M$) such that
     \[\vert\cfct(\lambda)\vert^{-2}\thicksim c_0\vert\lambda\vert^{2\alpha+1}\Biggl\{1+\sum_{j=1}^{M-1}c_j\lambda^{-j}+O\bigl(\lambda^{-M}\bigl)\Biggr\}\text{ as } \vert\lambda\vert\to\infty.\]
\item Let $\dfct(\lambda)=\vert\cfct(\lambda)\vert^{-2}$,
$\lambda\geq 0$, and $k\in\N_0$. There exists a constant
$c_k=c_{k,\alpha,\beta}$ such that
\[\Bigl|\frac{d^k}{d\lambda^k}\dfct(\lambda)\Bigr|\leq
c_k(1+\vert\lambda\vert)^{2\alpha+1-k}.\]
\item $\cfct'(\lambda)\thicksim \cfct(\lambda)O(\lambda^{-1})$ and $\cfct''(\lambda)\thicksim\cfct(\lambda)O(\lambda^{-2})$.
\end{enumerate}
\end{lemma}

Let $d\nu(\lambda)=d\nu_{\alpha,\beta}(\lambda)=(2\pi)^{-\frac{1}{2}}\vert\cfct(\lambda)\vert^{-2}\,d\lambda$ and denote by $L^p(d\nu)$ the associated weighted Lebesgue space on $\R^+$; note that $\cfct(\lambda)\cfct(-\lambda)=\cfct(\lambda)\overline{\cfct(\lambda)}=\vert\cfct(\lambda)\vert^2$ whenever $\alpha,\beta,\lambda\in\R$. The Jacobi transform, initially defined for, say
a function $f\in C_c^\infty(\R^+)$ by
\[\widehat{f}(\lambda)=\frac{\sqrt{\pi}}{\Gamma(\alpha+1)}\int_0^\infty f(t)\varphi_\lambda(t)\,d\mu(t)\]
extends to a unitary isomorphism from $L^2(d\mu)$ onto $L^2(d\nu)$,
and the inversion formula is the statement that
\[f(t)=\int_0^\infty\widehat{f}(\lambda)\varphi_\lambda(t)\,d\nu(\lambda)\]
holds in the $L^2$-sense, cf. \cite[Formula~4.5]{Koornwinder-newproof}. The limiting case $\alpha=\beta=-\frac{1}{2}$ is the
Fourier-cosine transform, which we will not study. One easily verifies that
$\widehat{\mathcal{L}f}(\lambda)=-(\lambda^2+\rho^2)\widehat{f}(\lambda)$.
Let $\Omega_p=\bigl\{\lambda\in\C\,:\, \vert\Im\lambda\vert<\bigl(\frac{2}{p}-1\bigr)\rho\bigr\}$, $p\in[1,2)$.

\begin{lemma}\label{lemma.Lq}
Assume $p\in[1,2)$ and $\lambda\in \Omega_p$. It follows that
$\varphi_\lambda\in L^q(d\mu)$, where $\frac{1}{p}+\frac{1}{q}=1$.
\end{lemma}
\begin{proof}See \cite[Lemma~3.1]{Koornwinder-FJ}.
\end{proof}

For later use we recall the following integral formula for the
Jacobi function $\varphi_\lambda^{(\alpha,\beta)}$ (cf.
\cite[Formula~2.21]{Koornwinder-newproof}), which is even valid
whenever $\Re\alpha>\Re\beta>-\frac{1}{2}$:
\begin{equation}\label{eqn.int.formula}
\varphi_\lambda^{(\alpha,\beta)}(t)=\frac{2}{\Delta_{\alpha,\beta}(t)}\int_0^t\cos(\lambda s) A_{\alpha,\beta}(s,t)\,ds
\end{equation}
where \begin{multline*}
A_{\alpha,\beta}(s,t)=\frac{2^{3\alpha+2\beta+\frac{1}{2}}\Gamma(\alpha+1)}{\Gamma(\alpha+\frac{1}{2})\Gamma(\frac{1}{2})} \sinh(2t)(\cosh t)^{\beta-\frac{1}{2}}\\
\times(\cosh t-\cosh s)^{\alpha-\frac{1}{2}} {_2}F_1\Bigl(\frac{1}{2}+\beta,\frac{1}{2}-\beta;\alpha+\frac{1}{2};\frac{\cosh t-\cosh s}{2\cosh t}\Bigr).
\end{multline*}
\begin{remark}\label{remark.GK}
For special values of $\alpha$ and $\beta$, determined by the root system of a
rank one Riemannian symmetric space, the functions $\varphi_\lambda$ are the
usual spherical functions of Harish-Chandra, and \eqref{eqn.int.formula}
reduces to \cite[Formulae~2.8 and 2.9]{Stanton-Tomas}. To be more precise
assume $G/K$ is a rank one Riemannian symmetric space of noncompact type, with
positive roots $\alpha$ and $2\alpha$. Furthermore let $p$ denote the
multiplicity of $\alpha$ and $q$ the multiplicity of $2\alpha$ (we allow $q$ to
be zero). With $\alpha:=\frac{1}{2}(p+q-1)$ and $\beta:=\frac{1}{2}(q-1)$ both real, and $p=2(\alpha-\beta)$ and $q=2\beta+1$, the
function $\varphi^{(\alpha,\beta)}_\lambda$ is precisely the usual elementary
spherical function $\varphi_\lambda$ as considered by Harish-Chandra, $\rho=\alpha+\beta+1=\frac{1}{2}(p+2q)$ as it should be, and $\text{dim}(G/K)=p+q+1=2(\alpha+1)$.

A similar choice of parameters $\alpha,\beta$ reveals that even spherical
analysis on Damek--Ricci spaces is subsumed by the present setup, see Section \ref{subsec.DM}.
\end{remark}

\section{A Transference Principle for Jacobi Convolution
Operators}\label{section.transference}
 Let us first recall from
\cite[Formula~(5.1)]{Koornwinder-FJ} the generalized translation $\tau_x$ of a suitable
function $f$ on $\R^+$, which is defined by
\[(\tau_xf)(y)=\int_0^\infty f(z)K(x,y,z)\,d\mu(z)\] where $K$ is an explicitly
known kernel function such that
\[
\varphi_\lambda(x)\varphi_\lambda(y)=\int_0^\infty\varphi_\lambda(z)K(x,y,z)\,d\mu(z).\]
In fact (cf. \cite[Formulae~(4.16),(4.19)]{Koornwinder-FJ}), for $\vert s-t\vert<u<s+t$,
\[\begin{split}
K(s,t,u)&=\frac{c_{\alpha,\beta}}{(\sinh s\sinh t\sinh u)^{2\alpha}}\int_0^\pi(1-\cosh^2s-\cosh^2t-\cosh^2u\\
&+2\cosh s\cosh t\cosh u\cosh y)_+^{\alpha-\beta-1}\sin^{2\beta}y\,dy\\
&=\frac{2^{\frac{1}{2}-\rho}\Gamma(\alpha+1)(\cosh s\cosh t\cosh u)^{\alpha-\beta-1}}
{\Gamma(\alpha+\frac{1}{2})(\sinh s\sinh t\sinh u)^{2\alpha}}\\
&\quad\times (1-B^2)^{\alpha-\frac{1}{2}}{_2}F_1\bigl(\alpha+\beta,\alpha-\beta;\alpha+\tfrac{1}{2};\tfrac{1}{2}(1-B)\bigr)
\end{split}\]
where $B(s,t,u)=\frac{\cosh^2s+\cosh^2t+\cosh^2u-1}{2\cosh s\cosh t\cosh u}$; elsewhere $K\equiv 0$. The associated generalized convolution product of two
functions $f,g\in L^2(d\mu)$ is defined by
\begin{equation}\label{eq.convolution}
f\star g(x)=\int_0^\infty
f(y)(\tau_xg)(y)\,d\mu(y) = \int_0^\infty\int_0^\infty f(y)g(z)K(x,y,z)\,d\mu(z)\,d\mu(y).
\end{equation}
This convolution is
associative and distributive, and by \cite[Equation~(5.4)(iv)]{Koornwinder-FJ},
$\widehat{f\star g}(\lambda)=\widehat{f}(\lambda)\widehat{g}(\lambda)$. The
usual inequalities for convolutions continue to hold, as we have the following
general form of the Young inequality.

\begin{proposition}\label{prop.young.ineq}
Let $p,q,$ and $r$ be such that $1\leq p,q,r\leq\infty$ and
$\frac{1}{p}+\frac{1}{q}-1=\frac{1}{r}$. The convolution $f\star g$ of $f\in
L^p(d\mu)$ and $g\in L^q(d\mu)$ is then well-defined as a function in
$L^r(d\mu)$, and $\|f\star g\|_r\leq\|f\|_p\|g\|_q$.
\end{proposition}
\begin{proof}
See \cite[Theorem~5.4]{Koornwinder-FJ}.
\end{proof}
Important for the approach in \cite{Stanton-Tomas}, as well as later papers like \cite{Giulini-Mauceri-Meda.Crelle}, is a close connection between convolution operators on the symmetric space and (Euclidean) convolution operators acting on function on the Euclidean component $A$ in the Iwasawa decomposition of $G$. This method of `transferring' convolution operators - and norm estimates thereof - between different spaces was developed at length in \cite{CoifmanWeiss}, but the lack of structure theory in the more general setting of Jacobi analysis certainly rules out an immediate extension of \cite[Theorem~2.4]{CoifmanWeiss}. Instead one needs hypergroups for the proof.

\begin{proposition}[Transference]\label{prop.Coifman.Weiss}
Let $k$ be a $\mu$-integrable even function on $\R$ and assume Euclidean
convolution with $\Delta k$ is bounded on $L^s$. Then convolution with $k$
is a bounded operator on $L^s(d\mu)$.
\end{proposition}
\begin{proof}
In the present formulation, the result is to be found in \cite[Theorem~4.6, Corollary~4.11]{Gigante}, but we should point out that Gigante defines the measure $d\mu$ differently. The required changes to the proofs are straightforward, however.
\end{proof}
Let us stress that the convolution kernels are not required to be compactly supported, contrary to the classical result \cite[Theorem~2.4]{CoifmanWeiss}. Density of $C_c^\infty$ has indeed been incorporated into the proof of \cite[Theorem~4.6]{Gigante}, yielding the more general statement above.

\section{Jacobi Multipliers}\label{section.multipliers}
In the present section we introduce the notion of a multiplier for the Jacobi
transform. For some reason, this seems to have been neglected in the literature
although Jacobi analysis was definitely known to the experts working on
multiplier problems in the seventies. Be that as it may, we start out with the following natural definition.

\begin{definition}
Let $m$ be a bounded, measurable, even function on $\R$, and let $T_m$ be the
bounded linear operator defined for $f\in L^2(d\mu)$ by $\widehat{T_mf}(\lambda)=m(\lambda)\widehat{f}(\lambda)$, $\lambda\in\C$. The function $m$ is called an \emph{$L^p$-multiplier for the Jacobi transform},
with $p\in(1,\infty)$, if the operator $T_m$ extends from $L^2(d\mu)\cap
L^p(d\mu)$ to a bounded linear operator on $L^p(\R^+,d\mu)$, that is, if
there exists a constant $c_p<\infty$ such that
$\|T_mf\|_p\leq c_p\|f\|_p$ for all $f\in L^p(d\mu)$.
\end{definition}

\begin{remark}
If the parameters $(\alpha,\beta)$ are chosen so as to correspond to a rank one
symmetric space (see Remark \ref{remark.GK}), the notion of a Jacobi
multipliers coincides with the well-known notion of radial multipliers for the
spherical transform, see \cite{ClercStein}, \cite{Stanton-Tomas}, and
\cite{Anker-Annals}. If $\alpha,\beta$ are associated to Damek--Ricci spaces,
the Jacobi multipliers were studied in \cite{Anker-Damek-Yacoub}. The results
obtained below generalize those in \cite{Stanton-Tomas} and
\cite{Anker-Damek-Yacoub}, but are rank one in nature and do not generalize
multiplier results for higher rank symmetric spaces.
\end{remark} Let $\textrm{PW}(\C)$ denote the space of even, rapidly
decreasing, analytic functions on $\C$ of exponential type, that is, $f$
belongs to $\textrm{PW}(\C)$ if and only if $f$ is even and entire analytic on
$\C$ and there exist constants $A>0$, $K_n$ ($n\in\N_0$) such that $\vert f(\lambda)\vert\leq
K_n(1+\vert\lambda\vert)^{-n}e^{A\vert\Im\lambda\vert}$  for every $\lambda\in\C$, $n\in\N_0$. By
\cite[Theorem~3.4]{Koornwinder-newproof}, the Jacobi transform is a bijection
from $C_c^\infty(\R^+)$ onto $\text{PW}(\C)$.
\begin{lemma}
Let $m\in\textrm{PW}(\C)$. Then $T_mf=\kappa\star f$, where $\kappa=m^\vee$,
that is,
\[\kappa(t)=\int_0^\infty
m(\lambda)\varphi_\lambda(t)\,\frac{d\lambda}{\cfct(\lambda)\cfct(-\lambda)}.\] By
density, the conclusion remains valid if $m$ is merely a bounded, measurable,
even function on $\R$.
\end{lemma}
\begin{proof}
The statement follows from the identity $\widehat{\kappa\star
f}=\widehat{\kappa}\widehat{f}=m\widehat{f}$.
\end{proof}
The following necessary condition is typical for non-Euclidean multiplier theorems, as already observed in \cite{ClercStein}.

\begin{lemma}\label{lemma.necessary}
Assume that $T_m$ is bounded on $L^p(d\mu)$ for some $p\in[1,2)$. Then $m$
extends to an even, bounded, holomorphic function in the strip $\Omega_p$
(continuous on the boundary if $p=1$).
\end{lemma}
The symmetric space case is covered by \cite{ClercStein} (see also
\cite[Theorem~4.4]{Stanton-Tomas}), whereas the case of Damek--Ricci spaces
appeared in the form of \cite[Proposition~4.10]{Anker-Damek-Yacoub}. The proof
we present below is an easy adaptation of the latter.
\begin{proof}
By Lemma~\ref{lemma.Lq}, $\varphi_\lambda$ belongs to $L^{p'}(d\mu)$ for
all $\lambda\in \Omega_p$. Let $f\in C_c^\infty(\R^+)$ and notice that
\[
\int_0^\infty(T_m\varphi_\lambda)f(t)\,d\mu(t) = \int_0^\infty\varphi_\lambda(t)T_mf(t)\,d\mu(t)
= \widehat{T_mf}(\lambda) =
m(\lambda)\int_0^\infty\varphi_\lambda(t)f(t)\,d\mu(t).\]
Therefore $T_m\varphi_\lambda=m(\lambda)\varphi_\lambda$, and it follows that
$\vert m(\lambda)\vert \|\varphi_\lambda\|_{L^{p'}} = \|m(\lambda)\varphi_\lambda\|_{L^{p'}} = \|T_m\varphi_\lambda\|_{L^{p'}} \leq \|T_m\|_{L^{p'}\to L^{p'}} \|\varphi_\lambda\|_{L^{p'}}$ for all $\lambda\in\Omega_p$, that is, $\sup\{\vert m(\lambda)\vert\,:\,\lambda\in\Omega_p\}\leq
\|T_m\|_{L^{p'}\to L^{p'}}$.

As in \cite[Lemma~3.1]{Koornwinder-FJ}, we presently prove that $m$ extends to
a holomorphic function on $\Omega_p$, which is continuous on
$\overline{\Omega}_p$ whenever $p=1$. Indeed, by Fubini's Theorem and the
Cauchy integral formula for the holomorphic function
$\lambda\mapsto\varphi_\lambda(t)$,
\[\begin{split}
\widehat{f}(\lambda_0) & = \int_0^\infty f(t)\varphi_{\lambda_0}\,d\mu(t) =
\int_0^\infty f(t)\Bigl\{\frac{1}{2\pi
i}\oint_{\mathcal{C}}\frac{\varphi_\lambda(t)}{\lambda-\lambda_0}\,d\lambda\Bigr\}\,d\mu(t)\\
&=\frac{1}{2\pi i}\oint_{\mathcal{C}}\int_0^\infty
f(t)\frac{\varphi_\lambda(t)}{\lambda-\lambda_0}\,d\mu(t)\,d\lambda =
\frac{1}{2\pi
i}\oint_{\mathcal{C}}\frac{\widehat{f}(\lambda)}{\lambda-\lambda_0}\,d\lambda,
\end{split}\]
where $\mathcal{C}$ is a contour encircling $\lambda_0$ within $\Omega_p$. Thus
$\widehat{f}$ is holomorphic in $\lambda_0\in\Omega_p$, since, for every $h$
small enough that $\lambda_0+h$ remains in the convex domain $\Omega_p$,
\[
\frac{\widehat{f}(\lambda_0+h)-\widehat{f}(\lambda_0)}{h}=\frac{1}{2\pi
i}\oint_{\mathcal{C}}\frac{\widehat{f}(\lambda)}{(\lambda-\lambda_0-h)(\lambda-\lambda_0)}\,d\lambda \longrightarrow \frac{1}{2\pi
i}\oint_{\mathcal{C}}\frac{\widehat{f}(\lambda)}{(\lambda-\lambda_0)^2}\,d\lambda\text{
as } h\to 0.
\]
\end{proof}
Another decidedly non-Euclidean result pertains to the Kunze--Stein phenomenon for the convolution structure:
\begin{lemma}\label{lemma.KunzeStein}
Let $p\in[1,2)$. There exists a constant $c_p$ such that
\begin{enumerate}
\romannum \item  if $f\in L^2(d\mu)$ and $g\in L^p(d\mu)$, then $f\star g\in
L^2(d\mu)$ with $\|f\star g\|_2\leq c_p\|f\|_2\|g\|_p$; \item if $f,g\in
L^2(d\mu)$ and $\frac{1}{p}+\frac{1}{q}=1$, then $f\star g\in L^q(d\mu)$ with
$\|f\star g\|_q\leq c_p\|f\|_2\|g\|_2$.
\end{enumerate}
\end{lemma}

\begin{proof}
We recreate the beautiful proof from
\cite[Theorem~5.5(i)]{Koornwinder-FJ}, since it demonstrates a technique we shall employ in later sections. As for the first statement, since $\widehat{g}$ is well-defined
and holomorphic in every $\lambda_0\in\Omega_p$, it follows from H\"older's
inequality that
$\vert\widehat{g}(\lambda_0)\vert\leq\|g\|_p\|\varphi_{\lambda_0}\|_q$.
Moreover, according to
\cite[Remark~6]{Koornwinder-newproof}, $\varphi_\lambda$ has a Laplace
type integral representation,
\[\varphi_\lambda^{(\alpha,\beta)}(t) = c_{\alpha,\beta}\int_0^1\int_0^\pi\vert\cosh t+\sinh t\, re^{i\psi}\vert^{i\lambda-\rho}
\quad\times (1-r^2)^{\alpha-\beta-1}r^{2\beta+1}(\sin \psi)^{2\beta}\,d\psi\,dr,\]
for $t>0$, by means of which we infer that
$\vert\varphi_{\lambda_0}^{(\alpha,\beta)}(t)\vert\leq\varphi_{i\Im\lambda_0}^{(\alpha,\beta)}(t)$. For $f,g\in C_c^\infty(\R^+)$ it thus follows that
$\|g\star
f\|_2^2=\|\widehat{g}\widehat{f}\|_2^2\leq
\|\widehat{g}\|_\infty^2\|\widehat{f}\|_2^2\leq\|f\|_2^2\|g\|_p^2\|\varphi_0^{(\alpha,\beta)}\|_q^2\leq c_p\|f\|_2^2$, where
$c_p=\|g\|_p^2\|\varphi_0^{(\alpha,\beta)}\|_q^2$ is
finite. The assertion of (i) now follows by density. \medskip

Let $k\in L^p(d\mu)$ and take $f,g$ to be continuous, compactly supported
functions. Since the kernel $K(\cdot,\cdot,\cdot)$ is invariant under
permutations of the three arguments, it follows that
\[\biggl|\int f\star g(x)k(x)\,d\mu(x)\biggr|\leq \int\vert g(x)\vert(\vert
k\vert\star\vert f\vert)(x)\,d\mu(x),\] which is bounded by $\|g\|_2\| \vert
k\vert\star \vert f\vert\|_2\leq c_p\|g\|_2\|k\|_p\|f\|_2$ by the H\"older
inequality and (i) of the
present Lemma. By duality it follows that
\[\|f\star g\|_q=\sup\biggl\{\biggl|\int f\star g(x)k(x)\,d\mu(x)\biggr|\,:\,
k\in L^p(d\mu), \|k\|_p\leq 1\biggr\}\leq c_p\|f\|_2\|g\|_2.\] The assertion of
(ii) thus follows once more by density.
\end{proof}

\begin{corollary}\label{cor.ClercStein}
Let $\delta>0$ be fixed and assume $k$ is an even function belonging
simultaneously to all the spaces $L^r(d\mu)$ for $r\in (1,1+\delta)$. Then
$f\mapsto k\star f$ is a bounded operator on $L^s(d\mu)$ for all $s\in
(1,\infty)$.
\end{corollary}
This follows easily from Lemma \ref{lemma.KunzeStein} and interpolation, thus
generalizing \cite[Lemma~5.4]{Stanton-Tomas} to our setting.

\begin{theorem}\label{thm.hormander.multiplier}
Let $m$ be an even, holomorphic function on $\Omega_1$ that satisfies the
H\"ormander-type condition
\[\forall\lambda=x+iy\in\Omega_1, a=0,1,\ldots,N: \Bigl|\frac{d^a}{dx^a}m(x+iy)
\Bigr|\leq c_{a,y}(1+\vert x\vert)^{-a},\] where $N$ is the least integer
greater than or equal to $\alpha+\frac{3}{2}$.
\begin{enumerate} \romannum \item Then $m$ is an $L^p$-multiplier for the
Jacobi transform for $p\in (1,\infty)$. \item The operator $T_m$ is weak type
$(1,1)$, that is,
\[\mu(\{t\,:\,\vert T_mf(t)\vert>a\})\leq\frac{\|m\|_{\text{mult}}\|f\|_{L^1(d\mu)}}{a}\text{ for all }a>0,\]
where \[\|m\|_{\text{mult}} =\max_{0\leq i\leq N}\sup_{\lambda\in\Omega_1}(1+\vert\lambda\vert)^i\vert m^{(i)}(\lambda)\vert.\]
\end{enumerate}
\end{theorem}

This H\"ormander-type multiplier theorem generalizes
\cite[Theorem~5.1]{Stanton-Tomas} and also provides a proof of the multiplier
theorem for Damek--Ricci spaces that was suggested just before
\cite[Theorem~4.17]{Anker-Damek-Yacoub}. We have more to say on this in
Subsection \ref{subsec.DM} towards the end of the paper.

The proof follows that in \cite{Stanton-Tomas} closely but we should add that a different approach, based on the Abel transform, was adopted in \cite{Anker-Annals}. The advantage of working with the asymptotic expansion is that we can easily obtain other interesting results in harmonic analysis as well: Fractional integrals are covered in Section \ref{sec.fract}, multipliers that are not integrable at infinity in \cite{Johansen-exp2} (see also \cite[Section~3]{Giulini-Mauceri-Meda.Crelle}), and almost everywhere convergence of the inverse Jacobi transform in \cite{Johansen-disc}.

Surely we still in effect work on the kernel level
in the proof of our multiplier theorem, just as Anker did it, so we still have
to estimate effectively the local and global parts of the kernel. The local
part will be covered by results in Section \ref{section.local}, whereas the
global part will be handled with the help of estimates from Section
\ref{section.long-range}.

To be more precise, let $m$ satisfy a H\"ormander-type condition as in the
multiplier theorem above and let $\psi$ be a fixed smooth, even function on
$\R$ such that $0\leq\psi\leq 1$, $\psi(t)\equiv 1$ for $\vert t\vert\leq
R_0^{1/2}$, and $\psi(t)\equiv 0$ for $\vert t\vert\geq R_0$; the constant
$R_0$ will be specified later in Theorem \ref{thm.asymptoticEXP}.
Let $k_1=m^\vee\psi$ be the local part of the kernel $\kappa=m^\vee$ and
$k_2=m^\vee(1-\psi)$ the global part. We analyze $k_1$ in Section
\ref{section.local} and $k_2$ in Section \ref{section.long-range}. It turns out that $k_2$ is easy to handle. The local part $k_1$ is troublesome, but since convolution with $k_1$ will be realized as a convolution operator on a space of homogeneous type, standard covering arguments will establish the weak type $(1,1)$ bound.

\section{Local Analysis}\label{section.local}
In what follows, $J_\mu(z)$ is the usual Bessel function of order $\mu$ and
$\mathcal{J}_\mu(z)$ is the modified Bessel function defined by
$\mathcal{J}_\mu(z)=2^{\mu-1}\Gamma(\frac{1}{2})\Gamma(\mu+\frac{1}{2})z^{-\mu}J_\mu(z)$.
The present section is devoted to the proof of the multiplier theorem for the local part of the kernel. It relies on the following Jacobi function-analogue of \cite[Theorem~2.1]{Stanton-Tomas}, see also
\cite[Section~2]{Schindler} for similar results. For the real parameter-case see also \cite{Brandolini-Gigante}. For the statement we introduce the quantity $\Delta'(t)=(\sinh
t)^{\alpha+\frac{1}{2}}(\cosh t)^{\beta+\frac{1}{2}}$.

\begin{theorem}\label{thm.asymptoticEXP}
Assume $\alpha>\frac{1}{2}$, $\alpha >\beta>-\frac{1}{2}$, and that $\lambda$ belongs either to a compact subset of
$\C\setminus (-i\N)$ or a set of the form
\[D_{\varepsilon,\gamma}=\{\lambda\in\C\,:\, \gamma\geq\Im\lambda\geq -\varepsilon\vert\Re\lambda\vert\}\]
for some $\varepsilon,\gamma\geq 0$. There exist constants $R_0, R_1\in
(1,\sqrt{\frac{\pi}{2}})$ with $R_0^2<R_1$ such that for every $M\in\N$ and
every $t\in[0,R_0]$
\begin{eqnarray}
\label{eqn.expansion.full}\varphi_\lambda^{(\alpha,\beta)}(t) =
\frac{2\Gamma(\alpha+1)}{\Gamma(\alpha+\tfrac{1}{2})\Gamma(\tfrac{1}{2})}\frac{t^{\alpha+\frac{1}{2}}}{\Delta'(t)} \sum_{m=0}^\infty
a_m(t)t^{2m}\mathcal{J}_{m+\alpha}(\lambda t)
\\
\varphi_\lambda^{(\alpha,\beta)}(t) =
\frac{2\Gamma(\alpha+1)}{\Gamma(\alpha+\tfrac{1}{2})\Gamma(\tfrac{1}{2})}\frac{t^{\alpha+\frac{1}{2}}}{\Delta'(t)} \sum_{m=0}^M
a_m(t)t^{2m}\mathcal{J}_{m+\alpha}(\lambda t)+E_{M+1}(\lambda t),
\end{eqnarray}
where
\begin{equation}\label{thm.asymptoticEXP.a.ests}
a_0(t)\equiv 1\text{ and } \vert a_m(t)\vert\leq c_\alpha(t)R_1^{-(\Re\alpha+m-\frac{1}{2})}\text{ for all } m\in\N.
\end{equation}
Additionally, the error term $E_{M+1}$ is bounded as follows:
\begin{equation}\label{thm.asymptoticEXP.error}
\vert E_{M+1}(\lambda t)\vert \leq \begin{cases} c_Mt^{2(M+1)}&\text{if }\vert\lambda t\vert\leq 1\\
c_M t^{2(M+1)}\vert\lambda t\vert^{-(\Re\alpha+M+1)}&\text{if
}\vert\lambda t\vert>1.\end{cases}
\end{equation}
\end{theorem}

\begin{figure}[h]
\centering
\includegraphics{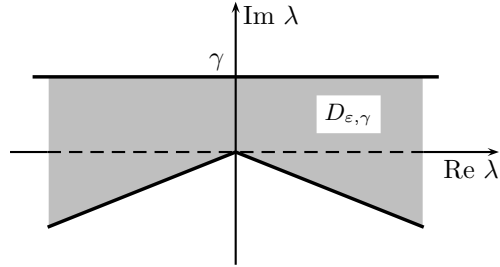}
\vspace{-5mm}
\caption{The set $D_{\varepsilon,\gamma}$.}
\end{figure}

\begin{remark}
The restriction on $\alpha$ is not the optimal one, as far as general Jacobi
analysis goes, but is needed for certain convergence arguments that are
important for this approach, and the restriction on $\lambda$ will also be
needed later, in Theorem \ref{thm.bigt}, where one needs a suitable analogue of
the classical Gangolli estimates. Another justification for imposing bounds on
$\Im\lambda$ is that the function $\varphi_\lambda$ is bounded precisely
when $\vert\Im\lambda\vert\leq\vert\Re\rho\vert$. Yet another one
is that in order to get multiplier theorems for the Jacobi transform, one has
to impose some such bound on $\Im\lambda$; see Lemma
\ref{lemma.necessary}.
\end{remark}

Let $\psi$ be an even, smooth function on $\R$ such that $0\leq\psi\leq 1$,
$\psi(t)\equiv 1$ whenever $\vert t\vert\leq R_0^{1/2}$, and $\psi(t)\equiv 0$
whenever $\vert t\vert\geq R_0$. Let $N$ be the least integer greater than or equal to
$\alpha+\frac{3}{2}$. From now on, whenever we use a function $m$, with certain
properties, we tacitly assume in addition that $m$ is rapidly decreasing.
Passage to the general case is then facilitated by standard techniques
involving approximate units. More specifically, let $h_t$ denote the heat
kernel for the Jacobi Laplacian $\mathcal{L}_{\alpha,\beta}$. By the inversion formula for the Jacobi transform, $e^{t\mathcal{L}_{\alpha,\beta}}f=h_t\star f$ for all $t>0$ and $f\in
L^2(d\mu)$, where $\star$ refers to the Jacobi convolution
\eqref{eq.convolution}. The ensuing operator semigroup is
ultracontractive, so it follows that convolution with $h_t$ is an approximate
unit.

\begin{proposition}\label{prop.L1-error} If $m\in C^N(\R)$ is even and
\begin{enumerate}
\romannum \item $D^am(0)=0$ whenever $0\leq a\leq N$, and \item $\vert
D^a_\lambda m(\lambda)\vert\leq c_\alpha(1+\vert\lambda\vert)^{-a}$ whenever
$0\leq a\leq N$,
\end{enumerate}
then there exists a function $e_0\in L^1(d\mu)$ such that
\[m^\vee(t)\psi(t)=c_0\psi(t)\frac{t^{\alpha+\frac{1}{2}}}{\Delta'(t)}\int_0^\infty
m(\lambda)\mathcal{J}_\alpha(\lambda t)\vert\cfct(\lambda)\vert^{-2}\,d\lambda
+ e_0(t).\]
\end{proposition}
\begin{remark}
The proof we are about to present remains valid, word for word, even when $\alpha$ and $\beta$ are allowed to be complex, but satisfying the added requirement that $\Re\alpha>\Re\beta+1$ (needed for an inversion formula). One should not use $\vert\cfct(\lambda)\vert^{-2}$ but rather $(\cfct(\lambda)\cfct(-\lambda))^{-1}$ as density, and one would have to make sense of the statement that the resulting function $e_0$ be integrable. As a matter of fact, if $\alpha$ and $\beta$ were to be complex, then $\Delta$ would be complex as well, with
\[\begin{split}
\Re\Delta_{\alpha,\beta}(t) & = \Delta_{\Re\alpha,\Re\beta}(t)\cdot \cos\bigl(2\Im\alpha\ln(2\sinh t)+2\Im\beta\ln(2\cosh t) \bigr)\\
\Im\Delta_{\alpha,\beta}(t) & = \sin\bigr(2\Im\alpha\ln(2\sinh t)+2\Im\beta\ln(2\cosh t)\bigr),
\end{split}\]
implying that
\[
\vert\Re\Delta_{\alpha,\beta}(t)\vert \leq\vert\Delta_{\Re\alpha,\Re\beta}(t)\vert \simeq\begin{cases} t^{2\Re\alpha+1}&\text{for }
t\lesssim 1\\
e^{2(\Re\rho)t}&\text{ for } t\gg 1\end{cases}\text{ and }
\vert\Im\Delta_{\alpha,\beta}(t)\vert \leq 1\text{ for all }t
\]
What the proof below will then show is that $e_0$ is integrable with respect to the real measure $\vert\Delta_{\Re\alpha,\Re\beta}(t)\vert dt$. Integrability with respect to $(\Im \Delta(t))dt$ would then follow at once. This easy remedy will \emph{not} yield complex parameter-analogues of the Young inequality, the Kunze--Stein phenomenon and H\"older's inequality, however, and this is precisely the reason why we cannot establish the multiplier theorem for complex parameters $\alpha,\beta$. We are grateful to Margit R\"osler for pointing out this problem.
\end{remark}

\begin{proof}
Choose $M=N$ in Theorem \ref{thm.asymptoticEXP} and define $e_0$ by
\begin{equation}\label{eqn.e0}
\begin{split}
e_0(t) &=
c_0\psi(t)\frac{t^{\alpha+\frac{1}{2}}}{\Delta'(t)}\sum_{m=1}^Nt^{2m}a_m(t)\int_0^\infty\mathcal{J}_{m+\alpha}(\lambda
t)m(\lambda)\vert\cfct(\lambda)\vert^{-2}\,d\lambda
\\
&\qquad+\psi(t)\int_0^\infty E_{N+1}(\lambda
t)m(\lambda)\vert\cfct(\lambda)\vert^{-2}\,d\lambda.\end{split}
\end{equation}
Furthermore let
\[\begin{split}
\varepsilon_m(t)&:=t^{2m}\int^\infty_0\mathcal{J}_{m+\alpha}(\lambda
t)m(\lambda)\vert\cfct(\lambda)\vert^{-2}\,d\lambda,\quad 1\leq m\leq N,\\
\varepsilon_{N+1}(t)&:=\int_0^\infty E_{N+1}(\lambda
t)m(\lambda)\vert\cfct(\lambda)\vert^{-2}\,d\lambda.
\end{split}\]
By \eqref{thm.asymptoticEXP.error} and Lemma \ref{lemma.precise-c},
\[\begin{split}
\vert\varepsilon_{N+1}(t)\vert&\leq
\|m\|_\infty\Bigl\{\int_{\{\lambda\,:\,\vert\lambda t\vert\leq 1\} }\vert
E_{N+1}(\lambda t)\vert\vert\cfct(\lambda)\vert^{-2}\,d\lambda +
\int_{\{\lambda\,:\,\vert\lambda t\vert\geq 1\} }\vert E_{N+1}(\lambda
t)\vert\vert\cfct(\lambda)\vert^{-2}\,d\lambda\Bigr\}\\
&\leq c_N\|m\|_\infty\Bigl\{\int_0^{\frac{1}{t}}
t^{2(N+1)}\vert\cfct(\lambda)\vert^{-2}\,d\lambda + \int_{\frac{1}{t}}^\infty
t^{2(N+1)}\vert\lambda t\vert^{-(\alpha+N+1)}\vert\cfct(\lambda)\vert^{-2}\,d\lambda)\Bigr\}\\
&\leq c_N\|m\|_\infty\Bigl\{\int_0^{\frac{1}{t}}t^{2(N+1)}\,d\lambda +
\int_{\frac{1}{t}}^\infty t^{2(N+1)}\vert\lambda t\vert^{-(\alpha+N+1)}(1+\vert\lambda\vert)^{2\alpha}\,d\lambda\Bigr\}
\end{split}.\]

The latter integral is convergent, since the combined power of
$\vert\lambda\vert$ in the integrand is roughly $\alpha-N-1<-2$,
whence $\vert\varepsilon_{N+1}(t)\vert\leq
c_N\|m\|_\infty(t^{2N+1}+t^{N+1-\alpha})$. Since $N+1-\alpha>0$, and $\psi$ is supported in a neighborhood around $0$, we conclude
that at least the function \[t\mapsto \psi(t)\int E_{N+1}(\lambda
t)m(\lambda)\vert\cfct(\lambda)\vert^{-2}\,d\lambda\] is integrable.

In order to show that the $N$ terms $t\mapsto \psi(t)\frac{t^{\alpha+\frac{1}{2}}}{\Delta'(t)}\varepsilon_m(t)$ in \eqref{eqn.e0} are integrable as well, we must proceed with more care, especially in regards to the term where $m=1$. First recall from \cite[p.~18]{Watson} that
$z^{-1}\frac{d}{dz}\mathcal{J}_{\mu-1}(z)=-c_{\mu-1}\mathcal{J}_\mu(z)$ for a
suitable constant $c_{\mu-1}$, and thus
$-c_\mu\mathcal{J}_\mu(z)=\frac{1}{z}\frac{d}{dz}(-\frac{1}{c_{\mu-1}}\frac{1}{z}\frac{d}{dz}\mathcal{J}_{\mu-2}(z))=-\frac{1}{c_{\mu-1}}(\frac{1}{z}\frac{d}{dz})^2\mathcal{J}_{\mu-2}(z)$,
and so on. Correspondingly write
$\mathcal{J}_{m+\alpha}(z)=c_N(z^{-1}\frac{d}{dz})^N\mathcal{J}_{m+\alpha-N}(z)$ for a suitable constant $c_N$. As a consequence of Lemma \ref{lemma.precise-c} it holds that
$\vert(\frac{d}{d\lambda}\circ\frac{1}{\lambda})^k\dfct(\lambda)\vert\simeq
(1+\vert\lambda\vert)^{2\alpha+1-2k}$. Since $m$ is rapidly decreasing by assumption, we infer that
$\vert(\frac{d}{d\lambda}\circ\frac{1}{\lambda})^k m(\lambda)\dfct(\lambda)\vert\simeq
(1+\vert\lambda\vert)^{2\alpha+1-2k}$. It is a straightforward exercise in bookkeeping and integration by parts to see that
\[\begin{split}
\varepsilon_m(t) & = c_k't^{2m}\int_0^\infty
m(\lambda)\vert\cfct(\lambda)\vert^{-2}\bigl(-\frac{1}{\lambda
t}\frac{d}{d(\lambda t)}\bigr)^k\mathcal{J}_{m+\alpha-k}(\lambda t)\,d\lambda\\
&=
c_k't^{2(m-k)}\int_0^\infty\mathcal{J}_{m+\alpha-k}(\lambda t)\bigl(\frac{d}{d\lambda}\circ\frac{1}{\lambda}\bigr)^k(m(\lambda)\dfct(\lambda))\,d\lambda,
\end{split}\]
for every integer $k$ (in particular $k=N$), yielding the estimate
\[
\vert\varepsilon_m(t)\vert \lesssim t^{2(m-N)}\int_0^\infty(1+\vert\lambda\vert)^{2\alpha+1-2N}\cdot 1\,d\lambda \leq t^{2(m-N)}\int_0^\infty(1+\vert\lambda\vert)^{-2}\,d\lambda\lesssim t^{2(m-N)}
\]
since $N\geq\alpha+\frac{3}{2}$ and $\vert\mathcal{J}_{m+\alpha-N}(\lambda t)\vert\lesssim 1$. For $t$ small we thus conclude that
\[\biggl|\psi(t)\frac{t^{\alpha+\frac{1}{2}}}{\Delta'(t)}\varepsilon_m(t)\vert\Delta(t)\vert\biggr|\lesssim t^{\alpha+\frac{1}{2}}t^{2(m-N)}t^{\alpha+\frac{1}{2}}\leq t^{2m-2},\]
which is integrable in a neighborhood of zero for $m\geq 1$. However, for $m=1$, we should not integrate by parts $N$ times if the resulting
modified Bessel function $\mathcal{J}_{m+\alpha-N}$ happen to end up with a negative order. For $m\geq 2$, however, the above calculations are fine, so it merely remains to consider the special case where $m=1$ and $N=1+\alpha$, $\alpha$ thus being an \emph{integer}. This is precisely what was considered towards the end of the proof of \cite[Proposition~4.1]{Stanton-Tomas} (cf. page 268, loc. cit), so we shall not repeat the easy argument.
The point is that for these parameters, there is a logarithmic blow-up of the integrand near zero which has to be compensated by estimating $\mathcal{J}_0(\lambda t)$ by $\vert\lambda t\vert^{-\frac{1}{2}}$. This finishes the proof that $e_0$ belongs to $L^1(d\mu)$.
\end{proof}
\begin{corollary}\label{cor.4.3}
Let $t\in[0,R_0]$. If $m\in C^N(\R)_{\text{even}}$ satisfies
\begin{itemize}
\item $D_\lambda^a m(0)=0$ whenever $0\leq a\leq N$, and \item $\vert
D_\lambda^a m(\lambda)\vert\leq c_a(1+\vert\lambda\vert)^{-a}$ whenever $0\leq
a\leq N$, \end{itemize} then
\[m^\vee(t) = c_0\int_0^\infty m(\lambda)\mathcal{J}_\alpha(\lambda
t)\vert\cfct(\lambda)\vert^{-2}\,d\lambda + \sum_{m=1}^Ne_m(t) + e(t),\] where
$\vert e(t)\vert\leq 1$, $\vert e_1(t)\vert\leq ct^{-2\alpha-1}$, and $\vert
e_m(t)\vert\leq ct^{2(m-1)-N}$ when $m>1$.
\end{corollary}

\begin{lemma}\label{lemma.Marcinkiewicz-on-R}
There are functions $\varepsilon_0\in L^1$ and $k_0$ bounded on $\R$
such that
\begin{enumerate}
\romannum \item $\Delta(t)k_1(t)=k_0(t)+\varepsilon_0(t)$ for $t\geq 0$, and
\item $k_0$ is continuously differentiable on all dyadic intervals
$I_j^-=(-2^{j+1},-2^j)$ and $I_j^+=(2^j,2^{j+1}), j\in\Z$ and satisfies
\[\sup_{j\in\Z} \int_{I_j^-\cup I_j^+}\vert k_0'(t)\vert\,dt<\infty.\]
\end{enumerate}
The function $k_0$ is therefore a classical multiplier on $\R$, and convolution
with the function $\Delta k_1$ thus a bounded operator on $L^s$ for
$s\in(1,\infty)$.
\end{lemma}
The statement regarding $k_0$ differs slightly from the analogue in
\cite[Lemma~5.3]{Stanton-Tomas}; since the $\alpha,\beta$ might not be integers, the proof given in \cite{Stanton-Tomas} does not carry over.

\begin{proof}
Let $\psi$ be an even, smooth function on $\R$ such that $0\leq\psi\leq 1$,
$\psi(t)\equiv 1$ when $\vert t\vert\leq R_0^{1/2}$, and $\psi(t)\equiv 0$ when
$\vert t\vert\geq R_0$. Additionally, let $\Phi$ be an even, smooth function on
$\R$ such that $0\leq\Phi\leq 1$, $\Phi(\lambda)\equiv 1$ when
$\vert\lambda\vert>2$ and $\Phi(\lambda)\equiv 0$ when $\vert\lambda\vert<1$.
Then
\[
k_1(t)=\psi(t)\int_0^\infty\Phi(\lambda)\varphi_\lambda(t)m(\lambda)\vert\cfct(\lambda)\vert^{-2}\,d\lambda
+\psi(t)\int_0^\infty(1-\Phi(\lambda))\varphi_\lambda(t)m(\lambda)\vert\cfct(\lambda)\vert^{-2}\,d\lambda,
\]
where the second integral is bounded by
\[\psi(t)\int_0^2\vert\varphi_\lambda(t)\vert\vert
m(\lambda)\vert\vert\cfct(\lambda)\vert^{-2}\,d\lambda
\lesssim\psi(t)\int_0^2\vert
m(\lambda)\vert\vert\cfct(\lambda)\vert^{-2}\,d\lambda\lesssim\psi(t),\] which
is in $L^1(d\mu)$.
\medskip

In order to construct the  functions $k_0$ and $\varepsilon_0$, we first
observe that $\Phi m$ satisfies the hypotheses in Proposition
\ref{prop.L1-error}. Indeed, $\Phi m$ is smooth and even, and the derivatives
all vanish in $0$. As for estimating the derivatives, the desired bound is
trivially true whenever $\vert\lambda\vert<1$, since $\Phi$ and all its
derivatives vanish identically. The estimate for $\vert\lambda\vert
>2$ is just the bound for $m$, since $\Phi'\equiv 0$. In the region
$\{1\leq\vert\lambda\vert\leq 2\}$ we bound the derivative of $\Phi m$ by its maximum. Now take
\[k_0(t) =
c_0\psi(t)\Delta(t)\frac{t^{\alpha+\frac{1}{2}}}{\Delta'(t)}\int_0^\infty\Phi(\lambda)m(\lambda)\mathcal{J}_\alpha(\lambda
t)\vert\cfct(\lambda)\vert^{-2}\,d\lambda\text{ and } \varepsilon_0(t)=\psi(t)\Delta(t)e_0(t),
\]
with $e_0\in L^1(d\mu)$ as in Proposition \ref{prop.L1-error}. It follows
that $\|\varepsilon_0\|_{L^1}\leq \|e_0\|_{L^1(d\mu)}$.
\medskip

We first address the estimates for $I_j^+$. Since $\psi$ and $\psi'$ are
supported in $[-2,2]$, we may assume that $j\leq 1$ in what follows. Moreover,
$\frac{d}{dt}\mathcal{J}_\alpha(\lambda
t)=\frac{\lambda}{t}\frac{d}{d\lambda}\mathcal{J}_{\alpha}(\lambda t)$, so
\begin{multline*}
\frac{dk_0}{dt} =
c_0\frac{d}{dt}(\psi(t)\Delta'(t)t^{\alpha+\frac{1}{2}})\int_0^\infty\Phi(\lambda)m(\lambda)\mathcal{J}_\alpha(\lambda
t)\vert\cfct(\lambda)\vert^{-2}\,d\lambda\\
+c_0\psi(t)\Delta'(t)t^{\alpha-\frac{1}{2}}\int_0^\infty\Phi(\lambda)\lambda
m(\lambda)\Bigl(\frac{d}{d\lambda}\mathcal{J}_\alpha(\lambda
t)\Bigr)\vert\cfct(\lambda)\vert^{-2}\,d\lambda.
\end{multline*}
In the first integral we estimate $\mathcal{J}_\alpha(\lambda t)$ by a constant
independent of $t$ (note that $\lambda t$ is now a real number, so that the
usual estimates for $\mathcal{J}_\alpha$ apply). The first half is therefore
bounded on $I_j^+$, with a bound independent of $j$.

As for the second half, we first recall that $\alpha>\frac{1}{2}$ by
assumption, so the factor in front of the integral is trivially bounded on all
$I_j$, with a bound independent of $j$. To estimate the second integral we could once more try and
integrate by parts, noting that, by assumption, $D^am(0)=0$ for
$a=0,1,\ldots N$ ($N$ being the least integer greater than or equal to $\alpha+\frac{3}{2}$) and $m$ is rapidly decreasing. Since
$\vert\frac{d^k}{d\lambda^k}\vert\cfct(\lambda)\vert^{-2}\vert\lesssim(1+\vert\lambda\vert)^{2\alpha+1-k}$ by Lemma \ref{lemma.precise-c},  and $\mathcal{J}_\alpha(\lambda t)=c_kt^{-2k}\bigl(\frac{1}{\lambda}\frac{d}{d\lambda}\bigr)^k\mathcal{J}_{\alpha-k}(\lambda t)$ for a suitable constant $c_k$, we would arrive at an estimate of
the form
\[\begin{split}
&\biggl|\int_{I_j^+}
\psi(t)\Delta'(t)t^{\alpha-\frac{1}{2}}\int_0^\infty\Phi(\lambda)\lambda
m(\lambda)\Bigl(\frac{d}{d\lambda}\mathcal{J}_\alpha(\lambda
t)\Bigr)\vert\cfct(\lambda)\vert^{-2}\,d\lambda\,dt\biggr|\\
\lesssim & \biggl| \int_{I_j^+}\psi(t)\Delta'(t)t^{\alpha-2k-\frac{1}{2}}\int \frac{d}{d\lambda}\bigl\{\Phi(\lambda)\lambda m(\lambda)\bigr\} \bigl(\frac{1}{\lambda}\frac{d}{d\lambda}\Bigr)^k\mathcal{J}_{\alpha-k}(\lambda t)\,d\lambda\,dt\biggr|\\
\lesssim& \int_{I_j^+}\psi(t)\Delta'(t)t^{\alpha-2k-\frac{1}{2}}\int_0^\infty \vert\mathcal{J}_{\alpha-k}(\lambda t)\vert \Bigl|\Bigl(\frac{1}{\lambda}\frac{d}{d\lambda}\Bigr)^k\frac{d}{d\lambda}(\Phi(\lambda)\lambda m(\lambda)\vert\cfct(\lambda)\vert^{-2} \Bigr|\,d\lambda\,dt\\
\lesssim & \int_{I_j^+}\psi(t)\Delta'(t)t^{\alpha-2k-\frac{1}{2}}\,dt \lesssim     \int_{I_j^+}\psi(t)t^{2\alpha-2k}\,dt
\end{split}\]
since $\vert\Delta'(t)\vert\simeq t^{\alpha+\frac{1}{2}}$ for small $t$. As this quantity is supposed to be uniformly bounded in $j$, it is \emph{not} desirable to take $k$ large. Instead note that $\int_{I_j^+}\psi(t)t^{2\alpha}\,dt\leq \int_{I_1^+}\psi(t)t^{2\alpha}\,dt$ for all $j\in\Z$, $j\leq 1$, yielding the desired uniform bound in $j$ for integrating over intervals $I_j^+$.

The remaining estimates involve integrals over $I_j^-$, and it is here that we need to employ integration by parts $k$ times, with $k$ so large that $2\alpha-2k$ be negative. There is also the very special case of $\alpha$ being an integer, since the required estimates take a different form when we have to estimate $\mathcal{J}_0(\lambda t)$; this case was treated in the proof of \cite[Lemma~5.3]{Stanton-Tomas} but did not reveal how to treat more general parameters. For this reason we had to resort to a different type of proof.

Since  $\sup_{j\in\Z}\int_{I_j^-\cup I_j^+}\vert
k_0'(t)\vert\,dt<\infty$, the Marcinkiewicz multiplier theorem implies that
$k_0$ is an $L^s$-multiplier on $\R$ for all $s\in(1,\infty)$.
\end{proof}

\begin{proof}[Proof of the 'local part` of Theorem
\ref{thm.hormander.multiplier}] It thus follows from Proposition
\ref{prop.Coifman.Weiss} that convolution with $k_1$ is a bounded operator $T_{k_1}$ on
$L^s(d\mu)$ for all $s\in(1,\infty)$.

Moreover, $T_{k_1}$ is of weak type $(1,1)$. This will follow from general results on spaces of homogeneous type, so let us briefly explain this. Define $B(t,r)\subset\R^+$ by
\[B(t,r)=\begin{cases} [t-r,t+r]&\text{if }t>r\\ [0,t+r]&\text{if } t\leq r\end{cases},\]
so that $\mu(B(t,r))=\int 1_{B(t,r)}(s)\Delta(s)\,ds$. Clearly $1+\mu([0,r])\thicksim(\cosh r)^{2\rho}$, $\mu(B(t,r))\thicksim r(\cosh t)^{2\rho}\thicksim re^{2\rho t}$ for $r\leq 1$ and $t>2$, and $\mu(B(t,r))\thicksim rt^{2\alpha+1}$ for $r\leq 1, t\in[2r,2)$, $\mu(t,r)\thicksim\mu([0,r])$ for $r\leq 1, t\leq\min\{2r,2\}$. Additionally, $\mu(B(t,nr))\lesssim\mu(B(t,r)$ for $r\leq 1$, so the ``ball'' $B(1)=[0,1]$ is indeed a space of homogeneous type with respect to the weighted measure $d\mu(t)=\Delta(t)dt$ and replaces the set $U_1$ in \cite[Lemma~18]{Anker-Annals}. In analogy with the important Vitali covering lemma we have the following easy result: Fix a covering $\{B(t_i,r_r)\}$ of a measurable set $E\subset\R^+$ with $r_i\leq 1$ for all $i$. Then there exists a disjoint subcollection $\{B(t_j,r_j)\}$ such that
\[\mu(E)\lesssim\sum_{j=1}^\infty\mu(B(t_j,r_j)).\]

It now suffices to consider functions $f$ on $\R^+$ that are supported in $B(1)$, in which case
$T_{k_1}f$ is indeed a
convolution operator on a space of homogeneous type. Since
$\|\tau_{y}f\|_q\leq\|f\|_q$ for all $q\in[1,\infty)$, $y\geq 0$, by
\cite[Lemma~5.2]{Koornwinder-FJ}, the localized kernel $k_1$ trivially
satisfies the H\"ormander cancellation property (cf. \cite[Formula~(39)]{Anker-Annals}). In fact
\[\|\tau_yk_1-k_1\|_{L^1(d\mu)}=
\int_{\R^+}\vert\tau_yk_1(x)-k_1(x)\vert\,d\mu(x)\leq 2\|k_1\|_{L^1(d\mu)}\leq
C,\] where $C$ is  a fixed constant independent of $y$. It is easily seen that
$T_{k_1}$ is a bounded operator from $L^r(d\mu)$ to $L^s(d\mu)$ whenever
$\frac{1}{r}-\frac{1}{s}=1$, so a standard result due to Coifmann and Weiss, cf. \cite[Section~3]{CoifmanWeiss-book}, yields the weak type $(1,1)$ property of $T_{k_1}$. The same technique was adopted by Anker in the proof of \cite[Corollary~17]{Anker-Annals}.

\end{proof}
\begin{remark}
The above proof of the weak type $(1,1)$ property was inspired by
\cite{Nilsson-pq}, and \cite[Lemma~14]{Anker-Annals}. Obviously our kernel is
much better behaved than those considered by either Anker or Nilsson, so we do
not have to work with a dyadic decomposition of $k_1$ like they did. We can
even get a weak type $(1,q)$ estimate for $T_{k_1}$ with little additional
effort, but this seems more difficult to establish for the global part of the
kernel, $k_2$. It is likely that similar strong $(p,p)$ and weak $(1,1)$
results hold if the kernels are such that their boundary values along
the edges of $\Omega_1$ are in some $L^2$-Sobolev space, as in \cite{Anker-Annals}.
Weaker requirements than those enforced by Anker are considered in
\cite{Johansen-exp2}; these considerations appear separately as we have not yet
been able to prove weak type $(1,1)$ results. These were also not established
in \cite{Giulini-Mauceri-Meda.Crelle}, precisely the results of which we
generalize in \cite{Johansen-exp2}.
\end{remark}

\section{Global Analysis}\label{section.long-range}
We shall presently investigate the behavior of $\varphi_\lambda(t)$ as $t$
tends to infinity and use the result to show that convolution with the 'global`
piece of a kernel $m^\vee$ is a bounded operator on $L^s(d\mu)$ for
$s\in(1,\infty)$. As in the case of symmetric spaces, this investigation
requires sharp bounds on the $\cfct$-function, a close study of the
Harish-Chandra series for $\varphi_\lambda$, and an analogue of the Gangolli
estimates in the Jacobi setting. Recall that
$\varphi_\lambda(t)=\cfct(\lambda)e^{(i\lambda-\rho)t}\phi_\lambda(t)+\cfct(-\lambda)e^{(-i\lambda-\rho)t}\phi_{-\lambda}(t)$, where we now
formally expand $\phi_\lambda(t)$ as a power series (the
``Harish-Chandra series''),
\[\phi_\lambda(t)=\sum_{k=0}^\infty\Gamma_k(\lambda)e^{-2kt}.\] Since $\phi_\lambda$
is a solution to \eqref{eqn.eigeneqn}, the $\Gamma_k(\lambda)$ are
given recursively -- according to \cite[Formula~3.4]{Stanton-Tomas} -- by
$\Gamma_0(\lambda)\equiv 1$,
\begin{multline*}
(k+1)(k+1-i\lambda)\Gamma_{k+1} = (\alpha-\beta)\sum_{j=0}^k(\rho+2j-i\lambda)\Gamma_j \\
+  (\beta+\tfrac{1}{2})\sum_{j=1}^{[\frac{k+1}{2}]}(\rho+2(k+1-2j)-i\lambda)\Gamma_{k+1-2j}(\lambda),
\end{multline*}
where $[\frac{k+1}{2}]$ is the integer part of $\frac{k+1}{2}$. In fact,
$\Gamma_{k+1}=a_k\Gamma_k+\sum_{j=0}^{k-1}b_j^k\Gamma_j$, where (by \cite[Corollary~3.4]{Stanton-Tomas})
\[a_k=1+\frac{\alpha-\beta-1}{k+1}+
\frac{\alpha-\beta-1+\frac{1}{k+1}(\alpha(\alpha-1)-\beta(\beta-1)+1)}{k+1-i\lambda}
\]
and
\[b_j^k=(-1)^{k+j+1}\frac{2\beta+1}{k+1}\Bigl(1+\frac{\rho+2j-1}{k+1-i\lambda}\Bigr).\]

\begin{lemma}[Gangolli estimates]
Let $D$ be either a compact subset of $\C\setminus (-i\N)$ or a set of the form
$D=\{\lambda=\xi+i\eta\in\C\,\vert\,\eta\geq -\varepsilon\vert\xi\vert\}$ for
some $\varepsilon\geq 0$. There exist positive constants $K,d$ such that
\begin{equation}\label{eqn.Gangolli.est}
\vert\Gamma_k(\lambda)\vert\leq K(1+k)^d \text{ for all }k\in\Z_+, \lambda\in
D.
\end{equation}
\end{lemma}
\begin{proof}
See \cite[Lemma~7]{FJ}.
\end{proof}

It follows that the expansion for $\phi_\lambda(t)$ converges uniformly on sets
of the form $\{(t,\lambda)\in[c,\infty)\times D\}$, where $c$ is a positive
constant. More precisely, if $\lambda\in D$, and $c>0$ is fixed, we see that
\[\forall t\geq c:\vert\phi_\lambda(t)\vert \leq\sum_{k=0}^\infty K(1+k)^de^{-2kt} \lesssim
\sum_{k=0}^\infty (1+k)^de^{-2ck}\lesssim 1,\] that is, $\phi_\lambda(t)$ is
bounded uniformly in $\lambda\in D$ for $t\geq c>0$. We will take $c=R_0$ in
later applications. Since $\lambda\mapsto\phi_\lambda(t)$ is analytic in a
strip containing the real axis, it follows as in the proof of
\cite[Lemma~7]{Meaney-Prestini} that derivatives of $\phi_\lambda$ in $\lambda$
are bounded independently of $\lambda$ as well.

The asymptotic behavior of $\varphi_\lambda(t)$ as $t$ increases can now be
investigated. The result is formally the same as the analogues in \cite{Stanton-Tomas} and
\cite{Schindler}, and the proof will even work for complex parameters $\alpha,\beta$.

\begin{theorem}\label{thm.bigt}
\begin{enumerate}
\romannum \item For every $M\geq 0$, $0\leq m\leq M$ and $\lambda\in\C$ with $\Im\lambda\geq 0$, there exist polynomials $f_{l_m}$ in $\lambda$ of degree $m$ such that
\[\Gamma_k(\lambda)=\sum_{m=0}^M\gamma_m^k+E_{M+1}^k,\]
where $\gamma_m^k$ is a sum of terms $1/f_{l_m}$, and where
\[\gamma_m^k(\lambda)\vert\leq A\frac{\vert\rho\vert^me^{2k}}{\vert\Re\lambda\vert^m},
\quad \vert D_{\Re\lambda}^a\gamma_m^k\vert\leq 2^aA\frac{\vert\rho\vert^m e^{2k}}{\vert\Re\lambda\vert^{m+a}}, \text{ and }
\vert E_{M+1}^k\vert\leq A\frac{\vert\rho\vert^{M+1}e^{2k}}{\vert\Re\lambda\vert^{M+1}};
\]
the constant $A$ is independent of $M$ and $\lambda$.

\item Let $\Lambda_m(\lambda,t)=\sum_{j=0}^\infty\gamma_j^{m+j}(\lambda)e^{-2jt}$. There exists a function $\mathcal{E}_{M+1}$ such that, for every $M\geq 0$ and $t\geq R_0$, it holds for $\lambda\in\C$ with $\Im\lambda\geq 0$ that
\[\phi_\lambda(t)=\sum_{m=0}^\infty\Lambda_m(\lambda,t)e^{-2mt}=\sum_{m=0}^M\Lambda_m(\lambda,t)e^{-2mt}+e^{-2(M+1)t}\mathcal{E}_{M+1}(\lambda,t),\]
where
\[\vert D_\lambda^a D_t^b\Lambda_m\vert\leq 2^{a+b}A\frac{\vert\rho\vert^me^{2m}}{\vert\Re\lambda\vert^{m+a}}G_b(t)\text{ and } \vert D_t^b\mathcal{E}_{M+1}\vert\leq 2^bA\frac{e^{2(M+1)}\vert\rho\vert^{M+1}}{\vert\Re\lambda\vert^{M+1}}G_b(t),\]
with $G_k(t):=\sum_{j=0}^\infty j^ke^{2k(1-t)}$.
\end{enumerate}
\end{theorem}
\begin{proof}
The algebraic properties of the Harish-Chandra series are investigated in \cite[Section~3]{Stanton-Tomas}, along with the estimates in part (i) of the theorem, and it is an arduous (yet elementary) matter to redo the proofs for complex parameters $\alpha,\beta$ instead. The improved statement in (ii) via the presence of the exponential factor in $e^{-2(M+1)t}\mathcal{E}_{M+1}(\lambda,t)$, was established in \cite[Lemma~6]{Meaney-Prestini}, the proof of which may trivially be repeated.
\end{proof}
\begin{lemma}\label{lemma.constrct.Kepsilon}
If $m$ is an even, analytic function in $\Omega_1$ satisfying the estimate
$\vert D^a_xm(x+iy)\vert\leq c_{a,y}(1+\vert x\vert)^{-a}$  for $0\leq
a\leq N$, $x+iy\in\Omega_1,$, and $\varepsilon\in (0,1)$ for some constant $c_{a_y}$ not depending on $m$, there exist a constant $c_\varepsilon$ and a nonnegative function $K_\varepsilon\in
L^2(\R^+)$ such that
\begin{equation}\label{eqn.constrct.Kepsilon}
\vert (1-\psi(t))m^\vee(t)\vert\leq c_\varepsilon e^{-(1+\varepsilon)(\Re\rho)t}(1+K_\varepsilon(t))\text{ for all }t\geq 0.
\end{equation}
\end{lemma}
The proof is technically involved and fairly long, but involves no novel new
insights compared to the proof of its symmetric space-analogue
\cite[Equation~(4.7)]{Stanton-Tomas}; the main idea is to use Theorem
\ref{thm.bigt} and the Gangolli estimates \eqref{eqn.Gangolli.est}. Since the
usual issues with non-integer parameters $\alpha,\beta$ and complex $\lambda$ persist, we have decided to include the
proof with a few more details. As the proof even works for complex $\alpha,\beta$, we have decided to write the proof as such, although we shall merely need the statement for real parameters.

\begin{proof}
Observe that 
\[m^\vee(t)=\int_0^\infty
m(\lambda)\varphi_\lambda(t)(\cfct(\lambda)\cfct(-\lambda))^{-1}\,d\lambda = \int_\R
m(\lambda)\cfct(-\lambda)^{-1}e^{(i\lambda-\rho)t}\phi_\lambda(t)\,dt\] by the
inversion formula. The formal series
$\phi_\lambda(t)=\sum_k\Gamma_k(\lambda)e^{-2kt}$ converges
uniformly for $t\geq R_0^{1/2}>1$, so
\[\vert(1-\psi(t))m^\vee(t)\vert\leq(1-\psi(t))\sum_{k=0}^\infty\biggl|\int_\R
m(\lambda)\cfct(-\lambda)^{-1}\Gamma_k(\lambda)e^{(i\lambda-\rho)t}\,d\lambda\biggr|e^{-2kt}.\]
Since $\lambda\mapsto\cfct(-\lambda)^{-1}$ is holomorphic when $\Im\lambda\geq 0$, the integrand \[h:\lambda\mapsto
m(\lambda)\cfct(-\lambda)^{-1}\Gamma_k(\lambda)e^{(i\lambda-\rho)t}\] is  holomorphic
there as well, and we may employ the standard technique of changing the contour of integration, $\lambda+i0\to
\lambda+i\varepsilon\rho$, for some fixed $\varepsilon\in(0,1)$. This technique permeates much of harmonic analysis, of course. We refer the reader
to the proof of \cite[Proposition~5.1]{Helgason-PW} for details.

\begin{figure}[h]
\centering
\includegraphics{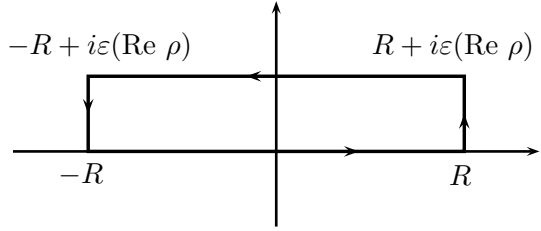}
\vspace{-5mm}
\caption{Change of contour of integration.}
\label{fig.contour}
\end{figure}

Let $\gamma_2$ denote the vertical line segment from $R$ to $R+i\varepsilon(Re\rho)$ in Figure \ref{fig.contour}, parameterized, say, by $\gamma_2(s)=R+i\varepsilon(Re\rho)s$, $0\leq s\leq 1$. Then
\[\int_{\gamma_2}h(z)\,dz=i\varepsilon(\Re\rho)e^{(iR-\rho)t}\int_0^1 m(R+i\varepsilon(\Re\rho)s)\cfct(-R-i\varepsilon(\Re\rho)s)^{-1}\Gamma_k(R+i\varepsilon(\Re\rho)s)e^{-\varepsilon(\Re\rho)st}\,ds.\]
Since $m$ decays rapidly as $R\to\infty$, by standing assumption, we infer from the estimates \eqref{eqn.c-fct.est} and \eqref{eqn.Gangolli.est} that $\vert\int_{\gamma_2}h\vert\to 0$ as $R\to\infty$. Analogously with the leftmost vertical line segment. By the Cauchy Theorem, it thus follows that
\begin{multline*}
\vert(1-\psi(t))m^\vee(t)\vert  \\
\lesssim (1-\psi(t))e^{-(1+\varepsilon)(\Re\rho)t}
\sum_{k=0}^\infty\biggl|\int_\R e^{i\lambda
t}\underbrace{m(\lambda+i\varepsilon\rho)\cfct(-\lambda-i\varepsilon(\Re\rho))^{-1}}_{:=q(\lambda,\varepsilon)}
\Gamma_k(\lambda+i\varepsilon(\Re\rho))\,d\lambda\biggr|e^{-2kt},\end{multline*}
whence it remains to establish the estimate
\begin{equation}\label{eqn.ce.plus.Ke}
\sum_{k=0}^\infty\biggl|\int_\R e^{i\lambda
t}q(\lambda,\varepsilon)\Gamma_k(\lambda+i\varepsilon(\Re\rho))\,d\lambda\biggr|e^{-2kt}\leq
c_\varepsilon+K_\varepsilon(t).
\end{equation}

To this end we fix a smooth even function $\Phi$ on $\R$ such that
$0\leq\Phi\leq 1$, $\Phi(\lambda)\equiv 1$ for $\vert\lambda\vert >2$ and
$\Phi(\lambda)\equiv 0$ for $\vert \lambda\vert <1$ and split the integral in \eqref{eqn.ce.plus.Ke} as

\[\int_\R e^{i\lambda t}(1-\Phi(\lambda))q(\lambda,\varepsilon)\Gamma_k(\lambda+i\varepsilon(\Re\rho))\,d\lambda +
 \int_\R e^{i\lambda t}\Phi(\lambda)q(\lambda,\varepsilon)\Gamma_k(\lambda+i\varepsilon(\Re\rho))\,d\lambda=:I_k+I\!I_k\]
The first integral is bounded, according the Gangolli estimates
\eqref{eqn.Gangolli.est}, by
\[\vert I_k\vert \leq
2\sup_{\lambda\in[-2,2]}\vert
q(\lambda,\varepsilon)\Gamma_k(\lambda+i\varepsilon(\Re\rho))\vert\leq c_\varepsilon
(1+k)^d\] for some constant $d$. But then
\[\sum_{k=0}^\infty\biggl|\int_\R e^{i\lambda t}(1-\Phi(\lambda))q(\varepsilon,\lambda)\Gamma_k(\lambda+i\varepsilon(\Re\rho))\,d\lambda\biggr|e^{-2kt}\leq
c_\varepsilon\sum_{k=0}^\infty (1+k)^d e^{-2kt}<c_\varepsilon.\]
\medskip

Estimating $I\!I_k$ is slightly more difficult, in part because we need good estimates for
$\cfct(-\lambda-i\varepsilon(\Re\rho))^{-1}$, which involves quotients of
Gamma functions with complex arguments. The estimates needed for controlling the $\Gamma_k$ are furnished
by Theorem \ref{thm.bigt}, however: Taking
$M=N$, $N$ being the least integer greater than or equal to $\Re\alpha+\frac{3}{2}$, the integral $I\!I_k$ is bounded according to
\begin{multline*}\vert II_k\vert
\leq  \sum_{j=0}^N\biggl|\int_\R\Phi(\lambda)e^{i\lambda
t}\gamma_j^k(\lambda+i\varepsilon(\Re\rho))q(\lambda,\varepsilon)\,d\lambda\biggr|\\
+ \sup_{\lambda}\vert m(\lambda+i\varepsilon\rho)\vert
\int_\R\Phi(\lambda)\vert E_{N+1}^k(\lambda)\vert
\vert\cfct(-\lambda-i\varepsilon(\Re\rho))\vert^{-1}\,d\lambda.\end{multline*}
We may safely assume that $\vert\lambda\vert>1$, by construction of $\Phi$, and
in this case \eqref{eqn.c-fct.est} implies that $\vert\cfct(-\lambda-i\varepsilon(\Re\rho))\vert^{-1}\leq
c_\varepsilon\vert\lambda\vert^{\Re\alpha+\frac{1}{2}}$. It follows from the
$E^k_{N+1}$-estimates in Theorem \ref{thm.bigt} that
\begin{multline*}
\sup_{\lambda}\vert m(\lambda+i\varepsilon(\Re\rho))\vert \int_\R\Phi(\lambda)\vert
E_{N+1}^k(\lambda)\vert
\vert\cfct(-\lambda-i\varepsilon(\Re\rho))\vert^{-1}\,d\lambda\\
\leq  c_\varepsilon
A\vert\rho\vert^{N+1}e^{2k}\int^\infty_1\vert\lambda\vert^{-(N+1)}\vert\lambda\vert^{\Re\alpha+\frac{1}{2}}\,d\lambda\leq c_\varepsilon e^{k},
\end{multline*}
where it was used that $N>\Re\alpha+1$. The total contribution from all error terms $E_{N+1}^k$  in \eqref{eqn.ce.plus.Ke} coming from the integrals $I\!I_k$ thus add up to a constant, since
\begin{multline*}
\sum_{k=0}^\infty \sup_{\lambda}\vert m(\lambda+i\varepsilon(\Re\rho))\vert
\int_\R\Phi(\lambda)\vert E_{N+1}^k(\lambda)\vert
\vert\cfct(-\lambda-i\varepsilon(\Re\rho))\vert^{-1}\,d\lambda\\
\leq c_\varepsilon(1-\psi(t))\sum_{k=0}^\infty e^{2kt} \leq
c_\varepsilon(1-\psi(t))\sum_{k=0}^\infty e^{k(1-R_0^{1/2})}\leq c_\varepsilon.
\end{multline*}

To complete the proof we need to settle the matter with the function
$K_\varepsilon$, so we simply define it to be whatever remains of
\eqref{eqn.ce.plus.Ke} to be estimated. More precisely, let
\[K_\varepsilon(t) =  \frac{1-\psi(t)}{\Delta'(t)}\sum_{j=0}^N\sum_{k=0}^\infty
e^{-2kt}\biggl|\int_\R\Phi(\lambda)e^{i\lambda
t}q(\lambda,\varepsilon)\gamma_j^k(\lambda+i\varepsilon(\Re\rho))\,d\lambda\biggr|\]
and
$f_j^k(\lambda)=\Phi(\lambda)q(\lambda,\varepsilon)\gamma_j^k(\lambda+i\varepsilon(\Re\rho))$.
Note that $K_\varepsilon(0)$ is well-defined and zero, and that $f_m^k(\lambda)=0$ for
$\lambda <1$. Also note that the $N$.th derivative of $f_j^k$ in with respect to $\lambda$ is
a sum of terms
\[\bigl\{(\tfrac{d}{d\lambda})^am(\lambda))\bigr\}
\bigl\{(\tfrac{d}{d\lambda})^b(\Phi(\lambda)\cfct(\lambda)^{-1})\bigr\}
\bigl\{(\tfrac{d}{d\lambda})^c\gamma_j^k(\lambda)\bigr\}\text{ with }
a+b+c=N.\] The estimates from Theorem \ref{thm.bigt}, combined with the assumption on $m$ and the usual $\cfct$-function estimates, yield
the estimate
\[\bigl|\bigl\{(\tfrac{d}{d\lambda})^am(\lambda))\bigr\}
\bigl\{(\tfrac{d}{d\lambda})^b(\Phi(\lambda)\cfct(\lambda)^{-1})\bigr\}
\bigl\{(\tfrac{d}{d\lambda})^c\gamma_m^k(\lambda)\bigr\}\bigr|\leq
c_\varepsilon\vert\lambda\vert^{-a}(1+\vert\lambda\vert)^{\Re\alpha+\frac{1}{2}-b}
e^{2k}\vert\lambda\vert^{-j-c}\] for $\vert\lambda\vert >1$, which is roughly of size $\vert\lambda\vert^{\Re\alpha+\frac{1}{2}-j-a-b-c}=\vert\lambda\vert^{\Re\alpha+\frac{1}{2}-j-N}$.
By the classical
Plancherel theorem on $\R$,

\[\begin{split}
\|K_\varepsilon\|_{L^2} &\leq
\sum_{j=0}^N\sum_{k=0}^\infty\biggl(\int_{R_0^{1/2}}^\infty e^{-4kt}\biggl|\int
e^{i\lambda t}f_j^k(\lambda)\,d\lambda\biggr|^2\,dt\biggr)^{\frac{1}{2}}\\
&\leq R_0^{-\frac{N}{2}}\sum_{j=0}^N\sum_{k=0}^\infty
e^{-2kR_0^{1/2}}\biggl(\int_\R t^{2N}\biggl|\int e^{i\lambda
t}f_j^k(\lambda)\,d\lambda\biggr|^2\,dt\biggr)^{\frac{1}{2}}\\
&=R_0^{-\frac{N}{2}}\sum_{j=0}^N\sum_{k=0}^\infty
e^{-2kR_0^{1/2}}\biggl(\int_\R\biggl|e^{i\lambda t}
\bigl(\frac{d}{d\lambda}\bigr)^Nf_j^k(\lambda)\,d\lambda\biggr|^2\,dt\biggr)^{\frac{1}{2}}\\
&=c\sum_{j=0}^N\sum_{k=0}^\infty
e^{-2kR_0^{1/2}}\biggl(\int_\R\Bigl|\frac{d^N}{d\lambda^N}f_j^k(\lambda)\Bigr|^2\,d\lambda\biggr)^{\frac{1}{2}}\\
&\leq c_\varepsilon\sum_{j=0}^N\sum_{k=0}^\infty
e^{-2kR_0^{1/2}}e^{2k}\biggl(\int_1^\infty\vert\lambda\vert^{2\Re\alpha+1-2j-2N}\,d\lambda\biggr)^{\frac{1}{2}}\\
&\leq c'_\varepsilon N\biggl(\int_1^\infty\vert\lambda\vert^{2\Re\alpha+1-2N}\,d\lambda\biggr)^{\frac{1}{2}}
\leq c'_\varepsilon N\biggl(\int_1^\infty\vert\lambda\vert^{-2}\,d\lambda\biggr)^{\frac{1}{2}}
\end{split}\]
since $N\geq\Re\alpha+\frac{3}{2}$. The integral $\int_1^\infty\vert\lambda\vert^{-2}\,d\lambda$ being finite, we have thus completed the proof.
\end{proof}

\begin{proposition}\label{prop.4.5}
Assume $\alpha,\beta\in\R$. If $m$ is an even, analytic function in $\Omega_1$ satisfying
\[\vert D^a_xm(x+iy)\vert\leq c_{a,y}(1+\vert x\vert)^{-a}\text{ for } 0\leq
a\leq N \text{ and all }x+iy\in\Omega_1,\] then $m^\vee(1-\psi)$ belongs to
$L^s(d\mu)$ for all $s\in (1,2)$.
\end{proposition}
\begin{proof}
Fix numbers $s\in (1,2)$ and $\varepsilon>\frac{2}{s}-1$, and let
$c_\varepsilon$ and $K_\varepsilon$ be as in Lemma
\ref{lemma.constrct.Kepsilon}. It then follows that
\[\|m^\vee(1-\psi)\|_s\leq c_\varepsilon\biggl(\int_0^\infty
e^{-(1+\varepsilon)s\rho t}\vert\Delta(t)\vert\,dt\biggr)^{\frac{1}{s}} +
c_\varepsilon\biggl(\int_0^\infty e^{-(1+\varepsilon)s\rho t}\vert
K_\varepsilon(t)\vert^s\vert\Delta(t)\vert\,dt\biggr)^{\frac{1}{s}},\]
where the first integral is finite, since $\vert\Delta(t)\vert\simeq
e^{2(\Re\rho) t}$ for large $t$. The second integral is bounded
according to the Euclidean H\"older inequality by
\[\biggl(\int_0^\infty e^{-(1+\varepsilon)\frac{2s}{2-s}\rho t}\vert\Delta(t)\vert\,dt\biggr)^{\frac{2-s}{2s}} \times
\biggl(\int_0^\infty\vert K_\varepsilon(t)\vert^2\,dt
\biggr)^{\frac{1}{2}}\] in which the second factor is finite by construction of
$K_\varepsilon$ and the first factor finite since the integrand is dominated by
an exponential function whose exponent is
$-(1+\varepsilon)\frac{2s}{2-s}\rho t+2\rho t= 2\rho t\frac{2-2s-\varepsilon s}{2-\varepsilon}$. This exponential function, in
turn, is integrable on $[0,\infty)$ since $\varepsilon>\frac{2}{s}-1$, $s\in
(1,2)$, so that $\frac{2-2s-\varepsilon s}{2-s}<-s<-1$.
\end{proof}

\begin{proof}[Proof of the 'global part` of Theorem
\ref{thm.hormander.multiplier}] Recall that we have fixed an even, smooth
function $\psi$ on $\R$ with $0\leq\psi\leq 1$, $\psi(t)\equiv 1$ when $\vert
t\vert\leq R_0^{1/2}$, and $\psi(t)\equiv 0$ when $\vert t\vert\geq  R_0$, and
that $k_2(t)=m^\vee(t)(1-\psi(t))$. According to Proposition
\ref{prop.4.5}, the function $k_2$ belongs to $L^r$ for every $r\in(1,2)$,
so Corollary \ref{cor.ClercStein} implies that $\|k_2\star f\|_s\leq
c_s\|f\|_s$ for $s\in(1,\infty)$.
\medskip

It remains to investigate the contribution of the convolution operator $T_{k_2}:f\mapsto k_2\star f$
to the weak type $(1,1)$ bound. Since our assumptions on the kernel guarantee integrability at infinity, as in \cite{Anker-Annals}, we can follow his strategy of reducing the estimates to an Euclidean estimate (cf. \cite[Proposition~5]{Anker-Annals}, in particular Equation (20), loc.cit). Alternatively, one may use the observation from \cite[Section~4]{Anker.duke} and \cite[Remark~2, p.125]{Stromberg}) that convolution against kernels satisfying a suitable estimate involving exponential decay in $\Re\lambda$ (as is indeed guaranteed to hold for our kernel $k_2$, due to Lemma \ref{lemma.constrct.Kepsilon} and its proof) gives rise to a weakly-$L^1$-bounded operator.
\end{proof}

\section{Fractional Integration}\label{sec.fract}
Let $\psi$ be an even, smooth function on $\R$ such that $0\leq\psi\leq 1$,
$\psi(t)\equiv 1$ whenever $\vert t\vert\leq R_0^{1/2}$, and $\psi(t)\equiv 0$
whenever $\vert t\vert\geq R_0$. Let $m_a(\lambda):=(\lambda^2+\rho^2)^{-\frac{a}{2}}$ for $a>0$ and let $k_a=m_a^\vee$. Since $k_a$ acts at least formally as ``fractional integration'' on even functions $f$ on $\R$ via $k_a\star f=-(-\mathcal{L})^{-\frac{a}{2}}f$, it is natural to try and establish an analogue of the famous Hardy--Littlewood--Sobolev theorem on fractional integration on $\R$. As above, we write $k_a$ as a sum $k_a=k_{1,a}+k_{2,a}$ where $k_{1,a}=k_a\psi$ and $k_{2,a}=k_a(1-\psi)$.

\begin{observation}\label{obs.ma}
Observe that $m_a$ belongs to $L^1(d\nu)$ if and only if $a>2(\alpha+1)$.
Indeed,
\[\biggl|\int_1^\infty
m_a(\lambda)\vert\mathbf{c}(\lambda)\vert^{-2}\,d\lambda\biggr|\lesssim
\int_1^\infty\vert\lambda\vert^{2\alpha+1-a}\,d\lambda,\] which is finite if
and only if $2\alpha+1-a<-1$, that is, when $a>2(\alpha+1)$.
\end{observation}

\begin{lemma}\label{lemma.tec.real}
Assume $\alpha>\frac{1}{2}$, $\alpha>\beta>-\frac{1}{2}$, and set $n_\alpha=2(\alpha+1)$.
\begin{enumerate}\romannum
\item\label{lemmar1} If $a=n_\alpha$ there exist finite constants $c_1,c_2$ such that $c_1\leq\vert k_{1,a}(t)/\log t\vert\leq c_2$ for all $t\geq 0$. In this case, $k_{1,a}$ belongs to $L^p(d\mu)$ if and only if $p\in[1,\infty)$.

\item\label{lemmar2} If $a\in(0,n_\alpha)$ there exist finite constants $c_1,c_2$ such that $c_1\leq\vert k_{1,a}(t)/t^{a-n_\alpha}\vert\leq c_2$ for all $t\geq 0$. In this case, $k_{1,a}$ belongs to $L^p(d\mu)$ if and only if $p\in[1,\frac{n_\alpha}{n_\alpha-a})$.

\item\label{lemmar3} $k_{2,a}$ belongs to $L^p(d\mu)$ if and only if $p\in(1,\infty]$.

\item\label{lemmar4} Assume $a>n_\alpha$ and $p\in [1,\infty]$. Then $k_{1,a}$ belongs to $L^p(d\mu)$.
\end{enumerate}
\end{lemma}

\begin{proof}
Statement \refr{lemmar3} follows easily from \refr{lemmar1} and \refr{lemmar2} as follows: If $k_{2,a}$ were integrable, then $\widehat{k}_{2,a}$ had to be continuous on the boundary of $\Omega_1$, according to Lemma \ref{lemma.necessary}. Since $\widehat{k}_{1,a}$ is always integrable by \refr{lemmar1} and \refr{lemmar2}, it would follow that $\widehat{k}_a=m_a$ were integrable, and therefore -- again by Lemma \ref{lemma.necessary} -- continuous on the boundary of $\Omega_1$. This is false, however, since $m_a$ has a singularity in the boundary point $\lambda=i\rho$.

On the other hand $k_{2,a}$ is even, analytic in $\Omega_1$ and satisfies the differential estimates of Proposition \ref{prop.4.5}, so $k_{2,a}$ belongs to $L^p(d\mu)$ for $p\in(1,2)$. Statement \refr{lemmar3} will thus follow from interpolation once we have established that $k_{2,a}$ is also in $L^\infty(d\mu)$. But according to Lemma \ref{lemma.constrct.Kepsilon}, one can bound $k_{2,a}=m_a^\vee(1-\psi)$ by
$\vert k_{2,a}(t)\vert\leq c_\varepsilon e^{-(1+\varepsilon)\rho t}(1+K_\varepsilon(t))$ for all $t\geq 0$,
where $K_\epsilon$ is a nonnegative $L^2$-function and $\varepsilon\in(0,1)$. The right hand side thus being essentially bounded on $\R$, we conclude that $k_{2,a}$ is in $L^\infty(d\mu)$.
\medskip

For the proof of \refr{lemmar4} we simply observe that $m_a$ is integrable if $a>n_\alpha$, by Observation \ref{obs.ma}, in which case $k_{1,a}$ is bounded compactly supported function and correspondingly in $L^p(d\mu)$ for all $p\in[1,\infty]$.
\medskip

In order to prove the remaining estimates in \refr{lemmar1} and \refr{lemmar2} we fix a, even, smooth function
 $\Phi$ on $\R$ such that $0\leq\Phi\leq 1$, $\Phi(\lambda)\equiv 1$ when
$\vert\lambda\vert>2$ and $\Phi(\lambda)\equiv 0$ when $\vert\lambda\vert<1$. Upon applying Corollary \ref{cor.4.3} to the function $m=\Phi m_a$, we may write $k_{1,a}$ in the form $k_{1,a}=m^\vee\psi+F$, where $F$ is a bounded remainder term, whence, by Corollary \ref{cor.4.3},
\[k_{1,a}(t)=c\int_0^\infty m_a(\lambda)\Phi(\lambda)\mathcal{J}_\alpha(\lambda t)\vert\mathbf{c}(\lambda)\vert^{-2}\,d\lambda + \sum_{m=1}^N e_m(t)+e(t)+F\text{ for } t\leq R_0\]
where $e$ is bounded, $\vert e_1(t)\vert\lesssim t^{-n_\alpha}$, and $\vert e_m(t)\vert\lesssim t^{2(m-1)-N}$ for $m\geq 2$ (where $N$ is the least integer greater than $\alpha+1$). Since $\vert\mathbf{c}(\lambda)\vert^{-2}\lesssim\vert\lambda\vert^{2\alpha+1}$ for $\vert\lambda\vert>2$ by Lemma \ref{lemma.precise-c}, the main contribution to the singularity of $k_{1,a}$ at $t=0$ thus comes from $t\mapsto \int_0^\infty m_a(\lambda)\Phi(\lambda)\mathcal{J}_\alpha(\lambda t)\vert\mathbf{c}(\lambda)\vert^{-2}\,d\lambda$ which we estimate by
$\int_0^\infty\mathcal{J}_\alpha(\lambda
t)(\lambda^2+\rho^2)^{-\frac{a}{2}}\lambda^{n_\alpha}\,d\lambda$. The latter integral can be calculated with the help of \cite[Formula~(20), p.24]{ErdelyiII}, indeed

\[\begin{split}
\int_0^\infty\mathcal{J}_\alpha(\lambda t)(\lambda^2+\rho^2)^{-\frac{a}{2}}\lambda^{2\alpha+1}\,d\lambda & =c\int_0^\infty J_\alpha(\lambda t)(\lambda t)^{-\alpha}(\lambda^2+\rho^2)^{-\frac{a}{2}}\lambda^{n_\alpha}\,d\lambda\\
&= ct^{-\alpha-\frac{1}{2}}\int_0^\infty J_\alpha(\lambda t)(\lambda^2+\rho^2)^{-\frac{a}{2}}\lambda^{\alpha+\frac{1}{2}}(\lambda t)^{\frac{1}{2}}\,d\lambda\\
&=c't^{-\alpha-\frac{1}{2}} \frac{e^{\alpha-(\frac{a}{2}-1)}t^{\frac{a}{2}-1+\frac{1}{2}}}{2^{\frac{a}{2}-1}\Gamma(\frac{a}{2})}K_{\alpha-(\frac{a}{2}-1)}(\rho t)\\
&=c''t^{\frac{a}{2}-\alpha-1}K_{\alpha-\frac{a}{2}+1}(\rho t),
\end{split}\]
where $K_\mu$ is a Bessel function of the third kind, of order $\mu$. We remind the reader that $K_\mu$ is defined as $K_\mu(z)=\frac{\pi}{2}\frac{I_{-\mu}(z)-I_\mu(z)}{\sin(\pi\mu)}$, where
\[I_\mu(z)=e^{-\frac{i\pi\mu}{2}}J_\mu(iz)=\sum_{k=0}^\infty\frac{1}{k!\Gamma(k+\mu+1)}\Bigl(\frac{z}{2}\Bigr)^{2k+\mu}.\]
Correspondingly,
\[K_\mu(z)=\frac{\pi}{2}\frac{1}{\sin(\pi\mu)}\sum_{k=0}^\infty\frac{(z/2)^{2k}}{k!}\biggl[\frac{1}{\Gamma(k-\mu+1)}\Bigl(\frac{z}{2}\Bigr)^{-\mu} - \frac{1}{\Gamma(k+\mu+1)}\Bigl(\frac{z}{2}\Bigr)^\mu\biggr],\]
by which it is seen that  $K_\mu(z)\thicksim 2^{\Re\mu-1}\Gamma(\mu)\vert z\vert^{-\Re\mu}$ as $z\to 0$ for $\Re\mu>0$, and $K_0(z)\thicksim\log (1/\vert z\vert)$ as $z\to 0$. Note that $\rho>0$ and $\alpha-1>-1$ by the standing assumption on $\alpha, \beta$. When $\alpha-\frac{a}{2}+1>0$, that is, when $a<n_\alpha$, we thus obtain the estimate
\[\vert k_{1,a}(t)\vert\thicksim \vert t^{\frac{a}{2}-\alpha-1}K_{\alpha-\frac{a}{2}+1}(\rho t)\vert \thicksim \vert t^{\frac{a}{2}-\alpha-1}t^{-(\alpha-\frac{a}{2}+1)}\vert = \vert t^{a-n_\alpha}\vert.\]
When $a=n_\alpha$, one has $\vert k_{1,a}(t)\vert\thicksim \vert K_0(\rho t)\vert\thicksim \vert\log(1/\vert t\vert)\vert = \vert\log t\vert$,
proving the estimates in \refr{lemmar1} and \refr{lemmar2}. The statements concerning integrability of $k_{1,a}$ now follow easily.
\end{proof}

\begin{theorem}\label{thm.frac}
Let $a>0$. The operator $I_a:f\mapsto k_a\star f$ is bounded from $L^p(d\mu)$ to $L^q(d\mu)$ if and only if either $p=q$ and $p\in(1,\infty)$, or $p<q$ and either
\begin{enumerate}
\romannum
\item\label{thmr1} $a>n_\alpha$, or
\item\label{thmr2} $a=n_\alpha$ and $q<\infty$, or
\item\label{thmr3} $0<a<n_\alpha$, and one of the following three conditions hold:
\begin{enumerate}
    \item $p>\frac{n_\alpha}{a}$;
    \item $1<p<\frac{n_\alpha}{a}$ and $\frac{1}{p}-\frac{a}{n_\alpha}\leq\frac{1}{q}$;
    \item $p=1$ and $1-\frac{a}{n_\alpha}<\frac{1}{q}<1$.
\end{enumerate}
\end{enumerate}
\end{theorem}

\begin{figure}
 \includegraphics[scale=0.9]{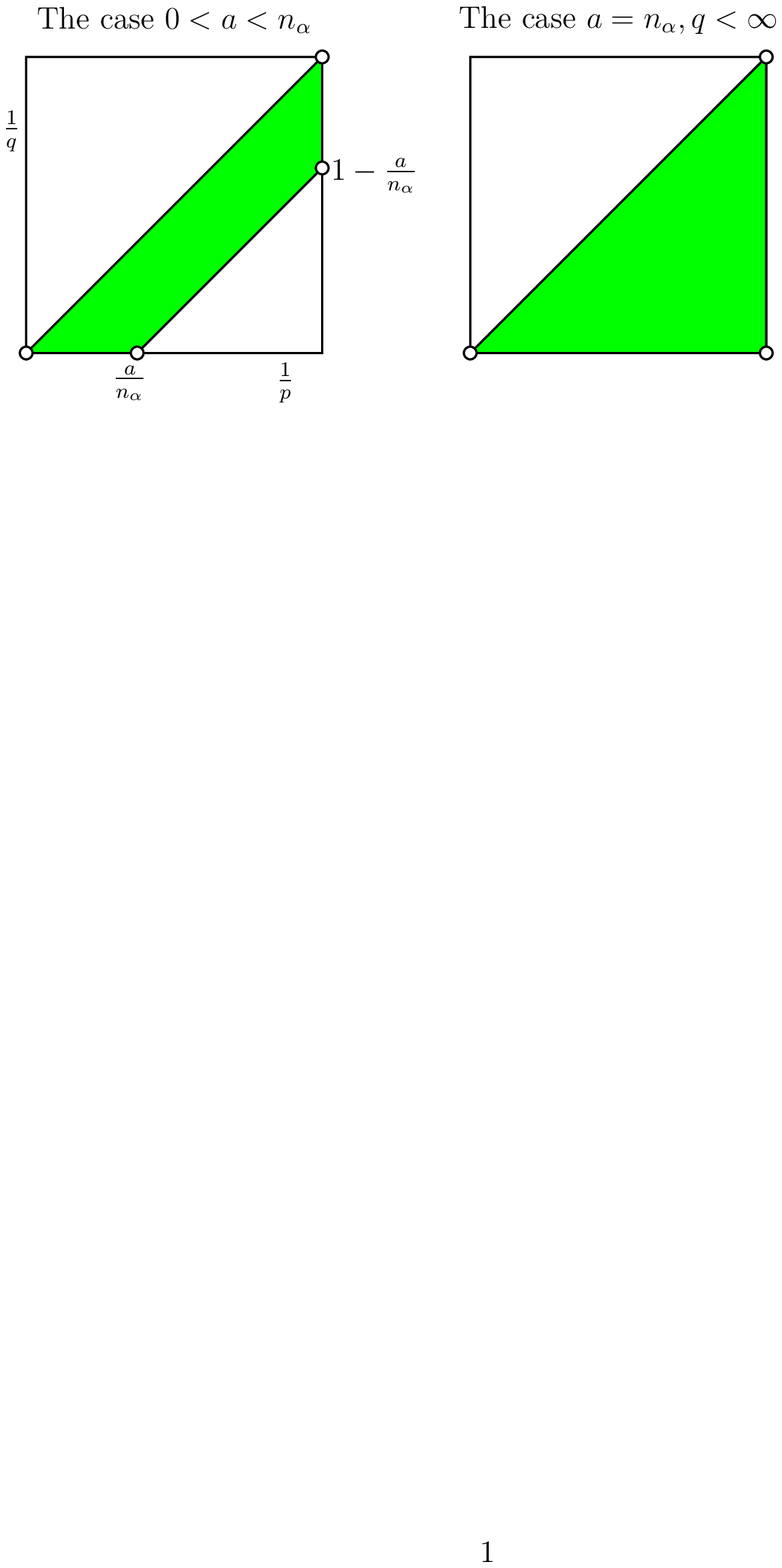}
 \caption{The Riesz potential is $L^p-L^q$ bounded in the solid regions.}
\end{figure}
\begin{proof}
First assume $p=q$. Presently both $k_{1,a}$ and $k_{2,a}$ are in $L^p(d\nu)$, with $p>1$, so the the convolution operator $f\mapsto k_a\star f$ is bounded on $L^p(d\mu)$ according to Theorem \ref{thm.hormander.multiplier}. The convolution operator cannot be bounded on neither $L^1$ nor $L^\infty$, however, since $L^1$- and $L^\infty$ multipliers for the Jacobi transform are continuous on $\overline{\Omega}_1$.
\medskip

Next suppose $p\neq q$. Convolution cannot be $L^p-L^q$-bounded unless $p<q$, so the case $p>q$ yields nothing and may be disregarded.

In case (i), it follows from Lemma \ref{lemma.tec.real} that $k_a$ belongs to $L^p$ whenever $p\in(1,\infty]$ and therefore defines an $L^p$-bounded operator, $\|k_a\star f\|_p\leq\|k_a\|_p\|f\|_1$ for $p\in(1,\infty]$. Moreover
 \[\vert k_a\star f(x)\vert\leq \int_0^\infty\vert k_a(y)\vert\vert\tau_xf(y)\vert\,d\mu(y)\leq \|k_a\|_{p'}\|\tau_xf\|_{p}\leq \|k_a\|_{p'}\|f\|_p,\]
 so that also $\|k_a\star f\|_\infty\leq\|k_a\|_{p'}\|f\|_p$ holds (this also follows from Young's inequality). Here $\|k_a\|_{p'}<\infty$ (since $p'\in(1,\infty)$ whenever $p\in(1,\infty)$. Note that we do not allow the possibility $p=\infty$ due to the requirement that $q>p$). In other words, $T_a$ is of strong type $(1,r)$ for all $r\in(1,\infty]$ and strong type $(s,\infty)$  for all $s\in[1,\infty)$. We wish to show that $T_a$ is therefore also strong type $(p,q)$, with $p$ and $q$ as stated in the theorem. This follows from the Riesz--Thorin interpolation theorem but it might be useful to explain how we tweak the interpolation parameters, as this point is somewhat confusing in the proof of \cite[Theorem~6.1]{Stanton-Tomas}. We thus seek a parameter $\theta\in(0,1)$ and particular choices for $r,s$ such that
 \[\frac{1}{p}=\frac{\theta}{1}+\frac{1-\theta}{r},\quad\text{ and } \quad \frac{1}{q}=\frac{\theta}{s}+\frac{1-\theta}{\infty}=\frac{\theta}{s}.\]
 This necessitates the choice $\theta:=s/q$, with $1\leq s<q$. If we assume for simplicity that $s=r'$, then $\theta=\frac{\frac{r}{r-1}}{q}=\frac{r}{q(r-1)}$, and $r$ must therefore solve the equation
 \[\frac{1}{p}=\frac{r}{q(r-1)}+\frac{1-\frac{r}{q(r-1)}}{r}=\frac{r^2+qr-q-r}{qr(r-1)},\]
 that is, we must take $r=\frac{pq}{q-p}$, which is indeed permissible since $p<q$ by assumption. It thus follows by interpolation that $T_a:f\mapsto k_a\star f$ is $L^p-L^q$ bounded.
\smallskip

In case (ii) it still holds that $k_a$ belongs to $L^p$ for $p\in(1,\infty)$ and we may repeat the reasoning used in case \refr{thmr1}, except that we must avoid using the inequality $\|k_{n_\alpha}\star f\|_\infty\leq\|f\|_1\|k_{n_\alpha}\|_\infty$ (which is still correct but useless since $\|k_{n_\alpha}\|_\infty=\infty$).
\smallskip

In case (iii) we use that $k_{2,a}$ will belong to $L^p$ for all $p\in(1,\infty]$, and in all of the cases $(a)$, $(b)$, and $(c)$. Convolution with $k_{2,a}$ is therefore $L^p-L^q$-bounded whenever $p<q$, by interpolation arguments identical so the ones above, and it thus remains to investigate boundedness of the convolution operator $T_{1,a}:f\mapsto f\star k_{1,a}$ in the three cases $(a)$, $(b)$, and $(c)$. It still holds that  $\|k_{1,a}\star f\|_r\leq\|k_{1,a}\|_r\|f\|_1$ and $\|k_{1,a}\star f\|_\infty\leq \|k_{1,a}\|_s\|f\|_{s'}$,
but $k_{1,a}$ fails to belong to $L^s(d\mu)$ for $s\geq\frac{n_\alpha}{n_\alpha-a}$, and Riesz--Thorin interpolation will therefore not give the full range of admissible $p$ and $q$. One can show, however, that convolution with $k_{1,a}$ is  weakly $L^1-L^r$ bounded and weakly $L^{r'}-L^\infty$ bounded, in the sense that $\|k_{1,a}\star f\|_{r,\infty}\leq c\|f\|_1$ and $\|k_{1,a}\star f\|_{\infty}\leq d\|f\|_{r'}$. The estimates follow from the Young inequality for weak $L^p$ spaces, to the effect that $\|k_{1,a}\star f\|_{s,\infty}\leq c_{p,r,s}\|f\|_p\|k_{1,a}\|_{r,\infty}$ for all $p,r,s\in[1,\infty]$ satisfying the relation $\frac{1}{s}+1=\frac{1}{p}+\frac{1}{r}$. Let us remind the reader that the weak $L^p$-norm $\|f\|_{s,\infty}$ of a measurable function $f$ is defined by
$\|f\|_{s,\infty}=\sup_{\gamma>0}\gamma d_f(\gamma)^{1/s}$, where $d_f(\alpha)=\mu(\{x\in\R^+\,\vert,\vert f(x)\vert>\alpha\})$. The Marcinkiewicz interpolation theorem for Lorentz spaces now yields the desired conclusion.
\medskip

To complete the proof we now simply observe that
estimates $\|k_{1,a}\star f\|_p\lesssim\|f\|_1$ and $\|k_{1,a}\star
f\|_\infty\lesssim\|f\|_{p'}$ break down whenever $p\geq \frac{n_\alpha}{n_\alpha-a}$.  It follows from the identity $\widehat{k_a\star
k_b}=\widehat{k_a}\widehat{k_b}=m_am_b=m_{a+b}=\widehat{k_{a+b}}$ and injectivity of the Jacobi transform that $k_a\star k_b=k_{a+b}$. If
either $p>1$ or $q<\infty$, we thus see that in order for
$\|k_a\star k_b\|_q\lesssim\|k_b\|_p$ to hold for some $(p,q)$, some $a<n_\alpha$, and
all $b>\frac{n_\alpha}{p'}$, the $a$, $p$, and $q$ have to be related as in Lemma
\ref{lemma.tec.real}. This precisely amounts to case (iii)(c), thereby finishing the proof.
\end{proof}
One may view the above result as a Jacobi analysis-analogue of \cite[Corollary~4.2]{Anker.duke} (specialized to Riesz potentials). It directly generalizes \cite[Theorem~6.1]{Stanton-Tomas}.

\section{Some Interesting Special Cases}\label{section.examples}
We now briefly explain how to specialize the results from previous sections in
order to obtain interesting multiplier results in more familiar settings.

\subsection{Damek--Ricci Spaces}\label{subsec.DM}
Let $\n$ be a two-step nilpotent Lie algebra with inner product
$\brac{\cdot}{\cdot}$ and associated norm $\|\cdot\|$, and let $\mathfrak{v}$
and $\z$ denote complementary orthogonal subspaces of dimension $m_\v$ and
$m_\z$, respectively, in $\n$ such that $[\n,\z]=\{0\}$ and $[\n,\n]\subset\z$.
Let $N$ denote the connected, simply connected Lie group with Lie algebra $\n$.
The algebra $\n$ (and by convention the group $N$) is of \emph{$H$-type} if for
every $Z$ in $\z$ the map $J_Z:\v\to\v$ defined by
$\brac{J_ZV}{V'}=\brac{Z}{[V,V']}$ satisfies the requirement that
$\|J_ZX\|=\|Z\| \|X\|$ for all $X\in\v$ and all $Z\in\z$. Upon identifying $N$
with its Lie algebra via $\v\times\z:N\to N$, $ (V,Z)\mapsto\exp(V+Z)$, the
group multiplication in $N$ becomes
\[(V,Z)(V',Z')=(V+V',Z+Z'+\frac{1}{2}[V,V']),\quad V,V'\in\v, Z,Z'\in\z.\]

The abelian group $A:=\R^+$ acts naturally by dilations $\delta_a:N\to N$,
$\delta_a(V,Z)=(a^{1/2}V,aZ)$, $(V,Z)\in N$, $a\in A$. As $\delta_a$ is an
automorphism, we may form the semi-direct product $S=N\rtimes A$. Recall that
the group multiplication on $S$ is defined by
\[(V,Z,a)(V',Z',a')=(V+a^{1/2}V',Z+aZ'+\tfrac{1}{2}a^{1/2}[V,V'],aa').\]
Fix a vector $H$ in $\a$ with the property that $\exp(tH)=e^t$ for all
$t\in\R$, and extend the inner product of $\n$ to $\mathfrak{s}=\n+\a$ by
demanding that $\n$ and $\a$ be orthogonal in $\s$ and that $H$ be a unit
vector. Regarded as a manifold with the natural left invariant Riemannian
metric, $S$ is  what has come to be known as a \emph{Damek--Ricci space}, due
to its prominent appearance in \cite{Damek-Ricci}. Specifically, the Lie
bracket on $\s$ is given by
\[[(V,X,a)(V',Z',a')]=(\frac{1}{2}aV'-\frac{1}{2}a'V,aZ'-a'Z+[V,V'],0);\]
the left-invariant metric is induced by
\[\brac{(V,Z,a)}{V',Z',a')} = \brac{V}{V'}+\brac{Z}{Z'}+aa',\]
and the associated left-invariant measure on $S$ is given by
$a^{-Q}\,dV\,dZ\,\frac{da}{a}$, where $Q=\frac{m_\v}{2}+m_\z$ is the
\emph{homogeneous dimension} of $N$. If we allow $N$ to be abelian, all
classical, noncompact Riemannian symmetric spaces of rank one are examples of
Damek--Ricci spaces.

The radial part of the Laplace--Beltrami operator on $S$ is given (in polar
geodesic coordinates) by
\begin{equation}\label{eq.radial-LB}
\mathcal{L}_r=\frac{\partial^2}{\partial
r^2}+\biggl(\frac{m_\v+m_\z}{2}\coth\frac{r}{2}+\frac{m_\z}{2}\tanh\frac{r}{2}\biggr)\frac{\partial}{\partial
r}.
\end{equation}
It was observed in \cite{Anker-Damek-Yacoub} (cf. Formula~(2.12), loc.cit.)
that $\mathcal{L}_r$ in fact coincides with the Jacobi operator
$\mathcal{L}_{\alpha,\beta}$, so any object that can be defined on $S$ by means
of spectral theory of $\mathcal{L}_r$ will have an analogue in Jacobi theory.
This might have motivated the definition of the spherical transform of a radial
function $f=f(r)$ on $S$ by
\[\widetilde{f}(\lambda) = \int_S \varphi_\lambda(x)f(x)\,dx = \frac{2^n\pi^{n/2}}{\Gamma(\tfrac{n}{2})}\int_0^\infty \sinh^{m_\v+m_\z}\bigl(\tfrac{r}{2}\bigr) \cosh^{m_\z}\bigl(\tfrac{r}{2}\bigr)\varphi_\lambda(r)f(r)\,dr.\]
Here $n=\dim~S=m_\v+m_\z+1$. The inversion formula for the spherical transform
is also familiar: If $f$ is radial and, say, in $C_c^\infty(S)$,
\[f(x)=\frac{2^{m_\z-2}\Gamma(\tfrac{n}{2})}{\pi^{\tfrac{n}{2}+1}}\int_0^\infty\varphi_\lambda(x)\widetilde{f}(\lambda)
\vert\mathbf{c}(\lambda)\vert^{-2}\,d\lambda,\quad\text{where
}\mathbf{c}(\lambda)=\frac{2^{Q-2i\lambda}\Gamma(\tfrac{n}{2})\Gamma(2i\lambda)}{\Gamma(i\lambda+\tfrac{Q}{2})
\Gamma(i\lambda+\tfrac{m_\v}{4}+\tfrac{1}{2})}.\] The spherical transform
extends uniquely to an isometry from the space $L^2(S)^\sharp$ of
square-integrable radial functions on $S$ onto
$L^2(\R_+,\vert\mathbf{c}(\lambda)\vert^{-2}d\lambda)$.

The strong type $(p,p)$-part of Theorem \ref{thm.hormander.multiplier} was
already established in \cite{Anker-Damek-Yacoub}, and while the experts surely
knew the weak type $(1,1)$ result, it did not appear in
\cite{Anker-Damek-Yacoub} nor elsewhere, as far as we know.

\subsection{Root Systems of Type $BC$}
Let $\a$ denote an $r$-dimensional real vector space with inner product
$(\cdot,\cdot)$, fix a root system $\Delta\subset\a^*$ together with a choice
of positive system $\Delta^+\subset\Delta$. Associate to $\lambda\in\a^*$ the
element $h_\lambda\in\a$ satisfying $\lambda(H)=(H,h_\lambda)$ for all
$H\in\a$, and define an inner product on $\a^*$ by
$(\lambda,\mu)=(h_\lambda,h_\mu)$. Set
$\a^+=\{H\in\a\,\vert\,\forall\alpha\in\Delta:\alpha(H)>0\}$. If $\lambda\neq
0$, we let $H_\lambda=\frac{2}{(\lambda,\lambda)}h_\lambda$ and observe that
$\alpha(H_\alpha)=2$. Let $r_\alpha(H)=H-\alpha(H)H_\alpha$ be the usual root
reflection and $W=\langle r_\alpha\rangle$ the associated Weyl group; it acts
on $\a^*$ and $\a_\C^*$ by $w\lambda(H)=\lambda(w^{-1}H)$.  A
\emph{multiplicity function} is any $W$-invariant function $m:\Delta\to\C$,
usually assumed to be $\R_+$-valued. Let $m_\alpha=m(\alpha)$ for notational
convenience.

We introduce the $r$-dimensional torus $A_\C=\a_\C/\Z\{i\pi
H_\alpha\,\vert\,\alpha\in\Delta\}$ along with the projection $\exp:\a_\C\to
A_\C$. Note that $A_\C=AT$, where $A=\a$ and $T=i\a/\Z\{i\pi
H_\alpha\,\vert\,\alpha\in\Delta\}$ is compact. Let
\[\begin{split}
A_\C^{\text{reg}}&=\exp\{H\in\a_\C\,\vert\,\forall\alpha\in\Delta:\alpha(H)\neq 0\}\\
A^{\text{reg}}&=A\cap A_\C^{\text{reg}} = \exp\{H\in\a\,\vert\,\forall\alpha\in\Delta:\alpha(H)\neq 0\}\\
A^+&=\exp\a^+\subset A^{\text{reg}}.
\end{split}\]

To $H\in\a$ one associates the directional derivative $\partial(H)$ defined by
$\partial(H)f(a)=\partial_tf(a\exp(tH))\vert_{t=0}$. Let $\{H_1,\ldots,H_r\}$
be any orthonormal basis of $\a$. The \emph{Heckman--Opdam Laplacian}
associated to $(\a,\Delta,m)$ is the $W$-invariant differential operator
\[L(m)=\underbrace{\sum_{j=1}^r\partial(H_j)^2}_{L_A} + \sum_{\alpha\in\Delta^+}m_\alpha\frac{1+e^{-2\alpha}}{1-e^{-2\alpha}}\partial(h_\alpha),\]
acting, say, on $C^\infty(\a)$. The operator $L_A$ is the usual Laplace
operator on $A$.

In the rank one situation, to which we now specialize, $\Delta^+=\{2,4\}$ with
root multiplicities $k_1=k(2)$ and $k_2=k(4)$, respectively. Set $\rho=k_1+2k_2$.
According to \cite[p.~89f]{Opdam}, the hypergeometric functions (the
construction of which is explained, for example, in \cite{Gestur_Henrik-SBT})
are then expressed by
\[F(\lambda,k,t):={_2}F_1\Bigl(\frac{\lambda+\rho}{2},\frac{-\lambda+\rho}{2},k_1+k_2+\frac{1}{2},-\sinh^2t\Bigr).\]
These are special types of Jacobi functions; with $\alpha=k_1+k_2-\frac{1}{2}$,
and $\beta=k_2-\frac{1}{2}$, one observes that
$F(i\lambda,k;t)=\varphi_\lambda^{(\alpha,\beta)}(t)$. The ideal situation
where $\alpha>\frac{1}{2}$, $\alpha>\beta>-\frac{1}{2}$
thus amounts to the requirement that $k_2>0$ and $k_1>1- k_2$.

\begin{figure}[h]
\centering
\includegraphics{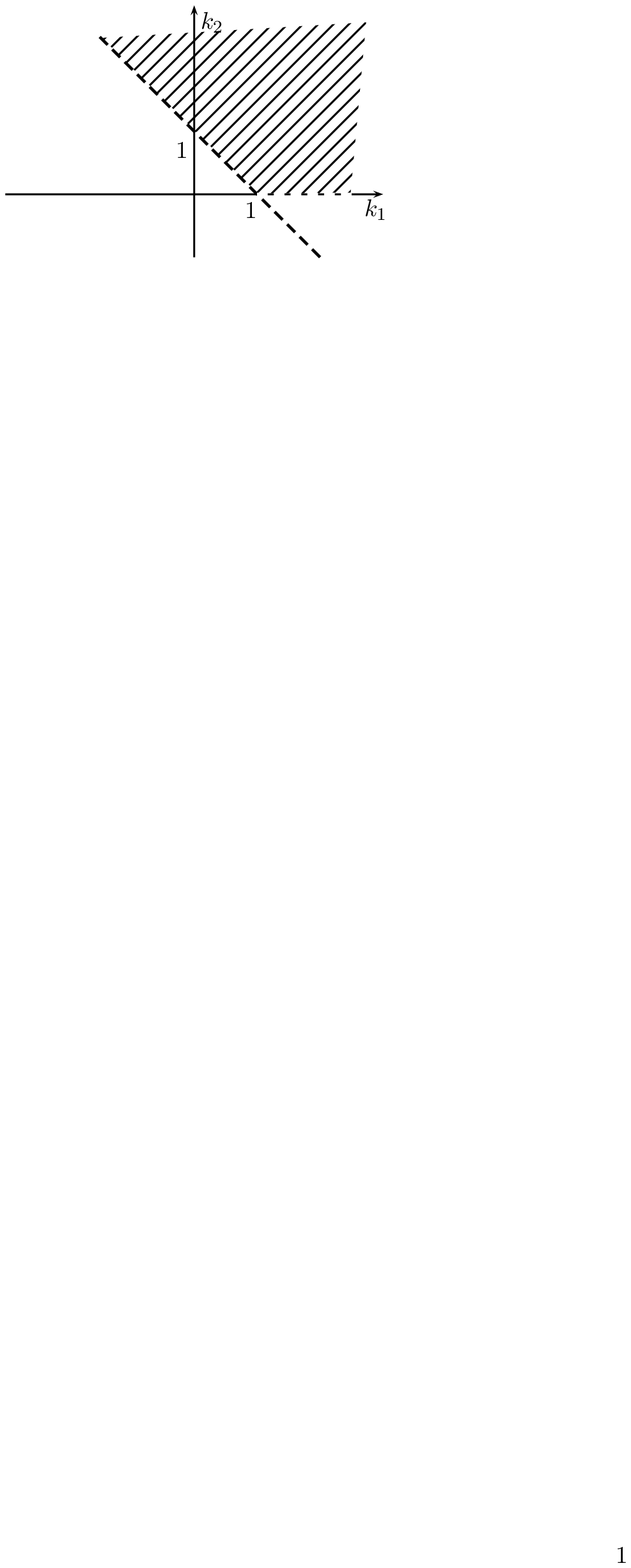}
\caption{Admissible range of $k_1$ and $k_2$.}
\end{figure}

The group $A$ is simply $\R$, and Weyl group invariance
of a function on $\R$ thus simply means that the function be even. The
$W$-invariant measures used by Opdam are then precisely our measures $d\mu$ and
$d\nu$, and the Heckman--Opdam transform -- where one integrates a function
against a hypergeometric function -- is simply the Jacobi transform for said
choice of parameters $\alpha,\beta$. Accordingly, the multiplier theorem is a
statement about $L^p$-multipliers for the ``hypergeometric Fourier transform'
associated to a rank one root system with complex multiplicity function. This
interpretation is amusing, at the least.

\appendix

\section{Asymptotic Analysis of Jacobi Functions}\label{app.asymptotic}
We presently remove the restriction on $\alpha,\beta$ that they be real, as the asymptotic analysis of Jacobi functions with complex parameters will be used elsewhere. Moreover, while the proofs from \cite{Stanton-Tomas} \emph{do} generalize fairly easily to the case of arbitrary \emph{real} parameters $\alpha,\beta$, it seems useful to still write out the details.

\begin{theorem}
Assume $\Re\alpha>\frac{1}{2}$, $\Re\alpha >\Re\beta>-\frac{1}{2}$, and that $\lambda$ belongs either to a compact subset of
$\C\setminus (-i\N)$ or a set of the form
\[D_{\varepsilon,\gamma}=\{\lambda\in\C\,:\, \gamma\geq\Im\lambda\geq -\varepsilon\vert\Re\lambda\vert\}\]
for some $\varepsilon,\gamma\geq 0$. There exist constants $R_0, R_1\in
(1,\sqrt{\frac{\pi}{2}})$ with $R_0^2<R_1$ such that for every $M\in\N$ and
every $t\in[0,R_0]$
\begin{eqnarray}
\label{A.eqn.expansion.full}\varphi_\lambda^{(\alpha,\beta)}(t) =
\frac{2\Gamma(\alpha+1)}{\Gamma(\alpha+\tfrac{1}{2})\Gamma(\tfrac{1}{2})}\frac{t^{\alpha+\frac{1}{2}}}{\Delta'(t)} \sum_{m=0}^\infty
a_m(t)t^{2m}\mathcal{J}_{m+\alpha}(\lambda t)
\\
\varphi_\lambda^{(\alpha,\beta)}(t) =
\frac{2\Gamma(\alpha+1)}{\Gamma(\alpha+\tfrac{1}{2})\Gamma(\tfrac{1}{2})}\frac{t^{\alpha+\frac{1}{2}}}{\Delta'(t)} \sum_{m=0}^M
a_m(t)t^{2m}\mathcal{J}_{m+\alpha}(\lambda t)+E_{M+1}(\lambda t),
\end{eqnarray}
where
\begin{equation}\label{A.thm.asymptoticEXP.a.ests}
a_0(t)\equiv 1\text{ and } \vert a_m(t)\vert\leq c_\alpha(t)R_1^{-(\Re\alpha+m-\frac{1}{2})}\text{ for all } m\in\N.
\end{equation}
Additionally, the error term $E_{M+1}$ is bounded as follows:
\begin{equation}\label{A.thm.asymptoticEXP.error}
\vert E_{M+1}(\lambda t)\vert \leq \begin{cases} c_Mt^{2(M+1)}&\text{if }\vert\lambda t\vert\leq 1\\
c_M t^{2(M+1)}\vert\lambda t\vert^{-(\Re\alpha+M+1)}&\text{if
}\vert\lambda t\vert>1.\end{cases}
\end{equation}
\end{theorem}

\begin{proof}
In order to expand the hypergeometric function that appears in
\eqref{eqn.int.formula} in a power series, it is helpful to start with the
series expansion of the functions $s\mapsto (\frac{\cosh t-\cosh
s}{t^2-s^2})^{d+j}$, $j\in\N_0$. The analysis is carried out on page 255 in
\cite{Stanton-Tomas} and has nothing to do with Jacobi functions, so we shall
be brief. First note that
\begin{equation}\label{eqn.F.expansion}
{_2}F_1\Bigl(\frac{1}{2}+\beta,\frac{1}{2}-\beta;\alpha+\frac{1}{2};\frac{\cosh
t-\cosh s}{2\cosh t}\Bigr) = \sum_{j=0}^\infty d_j\Bigl(\frac{\cosh t-\cosh
s}{2\cosh t}\Bigr)^j
\end{equation}
whenever $\vert\cosh t-\cosh s\vert< 2\vert\cosh t\vert$, where
\[d_j=\frac{\Gamma\bigl(\alpha+\frac{1}{2}\bigr)}{\Gamma(\frac{1}{2}+\beta)\Gamma(\frac{1}{2}-\beta)}
\frac{\Gamma(\frac{1}{2}+\beta+j)\Gamma\bigl(\frac{1}{2}-\beta+j\bigr)}
{\Gamma(\alpha+\frac{1}{2}+j)\Gamma(j+1)}.\] Observe that $\vert d_j\vert$ is
comparable to $j^{-(\Re\alpha+\frac{1}{2})}$, according to either one of
the classical estimates
\[\Gamma(z)=\Bigl(\frac{2\pi}{z}\Bigr)^{\frac{1}{2}} \Bigl(\frac{z}{e}\Bigr)^z\bigl(1+O(\vert z\vert^{-1})\bigr)\text{ or }
\frac{\Gamma(z+\alpha)}{\Gamma(z+\beta)}=z^{\alpha-\beta}(1+O(z^{-1}))\]
implying that $(d_j)$ is an absolutely summable sequence whenever $\Re\alpha>\frac{1}{2}$.

We recall from \cite[Proposition~2.3]{Stanton-Tomas} that
\begin{equation}\label{eq.cosh-expand}
\biggl(\frac{2\cosh t-2\cosh s}{t^2-s^2}\biggr)^z=\Bigl(\frac{\sinh
t}{t}\Bigr)^z\sum_{k=0}^\infty a_k(t,z)(t^2-s^2)^k,\quad z\in\C
\end{equation}
for suitable functions $a_k(\cdot,\cdot)$ (the convergence being uniform for
$\vert t^2-s^2\vert<3\pi^2$ and $\vert t\vert <\pi$), and that there exists a
number $R_1\in (1,\sqrt{\tfrac{\pi}{2}})$ such that for all $z\in\C$ with
$\Re z>0$ and $t\in[-\sqrt{R_1},\sqrt{R_1}]$, one has
\begin{equation}\label{eqn.ST.2.11}
\biggl|\biggl(\frac{\sinh t}{t}\biggr)^z a_k(t,z)\biggr|\leq\biggl(\frac{4\cosh
t}{R_1}\biggr)^{\Re z}R_1^{-k}.
\end{equation}
This estimate is proved in \cite{Stanton-Tomas} for $x>0$ but the proof
trivially goes through for complex numbers $z$ as well. Choose a positive
number $R_0$ with $1<R_0<\sqrt{R_1}$. Since $\vert\frac{\cosh t-\cosh s}{2\cosh
t}\vert<\frac{1}{2}$ for $t\in[0,R_0]$, $s\leq t< R_0<\sqrt{\pi}$, the series
\eqref{eqn.F.expansion} converges uniformly in this domain. It thus follows
from \eqref{eqn.int.formula} that
\[\begin{split}
&\frac{\Gamma(\alpha+\tfrac{1}{2})\Gamma(\tfrac{1}{2})}{2^{3\alpha+2\beta+\frac{3}{2}}\Gamma(\alpha+1)} \frac{\Delta(t)}{\sinh(2t)(\cosh t)^{\beta-\frac{1}{2}}}\varphi_\lambda(t) \\
=&\frac{\Gamma(\alpha+\tfrac{1}{2})\Gamma(\tfrac{1}{2})}{2^{\frac{1}{2}+\alpha}\Gamma(\alpha+1)} (\sinh t)^{2\alpha}(\cosh t)^{\beta+\frac{1}{2}}\varphi_\lambda(t)\\
=& \int_0^t\cos(\lambda s)(\cosh t-\cosh s)^{\alpha-\frac{1}{2}}{_2}F_1\Bigl(\tfrac{1}{2}+\beta,\tfrac{1}{2}-\beta;\alpha+\tfrac{1}{2};\tfrac{\cosh t-\cosh s}{2\cosh t}\Bigr)\,ds\\
=&\int_0^t\cos(\lambda s)(\cosh t-\cosh s)^{\alpha-\frac{1}{2}}\Bigl\{\sum_{j=0}^\infty d_j\bigl(\tfrac{\cosh t-\cosh s}{2\cosh t}\bigr)^j\Bigr\}\,ds\\
=&2^{\frac{1}{2}-\alpha}\sum_{j=0}d_j(4\cosh t)^{-j}\int_0^t\cos(\lambda s)(2\cosh t-2\cosh s)^{\alpha+j-\frac{1}{2}}\,ds
\end{split}\]
which can be rewritten by means of \eqref{eq.cosh-expand} as
\begin{multline*}
2^{\frac{1}{2}-\alpha}\sum_{j=0}^\infty d_j(4\cosh t)^{-j}\int_0^t\cos(\lambda s)(t^2-s^2)^{\alpha+j-\frac{1}{2}}\Bigl(\frac{\sinh t}{t}\Bigr)^{\alpha+j-\frac{1}{2}}
\Bigl\{\sum_{k=0}^\infty a_k\bigl(t,\alpha+j-\tfrac{1}{2}\bigr)(t^2-s^2)^k\Bigr\}\,ds\\
=2^{\frac{1}{2}-\alpha}\sum_{j=0}^\infty\sum_{k=0}^\infty d_j(4\cosh t)^{-j}\Bigl(\frac{\sinh t}{t}\Bigr)^{\alpha+j-\frac{1}{2}} a_k\bigl(t,\alpha+j-\tfrac{1}{2}\bigr) \int_0^t\cos(\lambda s)(t^2-s^2)^{j+k+\alpha-\frac{1}{2}}\,ds.
\end{multline*}

In order to compute the latter integral, we first note that the Bessel function
$J_{j+k+\alpha}$ is well-defined since $\Re(j+k+\alpha)>-\frac{1}{2}$. The
integral representation
\begin{equation}\label{egn.Bessel.intrep}
J_\mu(z)=\frac{(z/2)^\mu}{\Gamma(\mu+\frac{1}{2})\Gamma(\frac{1}{2})}\int_{-1}^1e^{izs}(1-s^2)^{\mu-\frac{1}{2}}\,ds
\end{equation}
yielding an absolutely convergent integral whenever $\Re\mu>-\frac{1}{2}$, may thus be employed. Indeed, Bessel functions of arbitrary
real exponent $k>-\frac{1}{2}$ are investigated in \cite[Chapter~IV,
Section~3]{SteinWeiss}, and they are defined for complex parameters through
analytic continuation, as described, for example, in \cite[Chapter~VI]{Watson}.
Notice that
\begin{equation}\label{eqn.standard.Bessel-estimate}
\vert J_\mu(z)\vert\leq\Bigl|\frac{z}{2}\Bigr|^{\Re\mu}\frac{e^{\vert\Im z\vert}}{\vert\Gamma(\mu+1)\vert}.
\end{equation}

Additionally, it follows that
\[\begin{split}
\int_0^t\cos(\lambda s)(t^2-s^2)^{j+k+\alpha-\frac{1}{2}}\,ds
&=t\int_0^1\cos(\lambda tr)(t^2-t^2r^2)^{j+k+\alpha-\frac{1}{2}}\,dr\\
&=t^{2(j+k+\alpha)}\int^1_0\cos(\lambda tr)(1-r^2)^{j+k+\alpha-\frac{1}{2}}\,dr\\
&=t^{2(j+k+\alpha)}\frac{\Gamma\bigl(j+k+\alpha+\frac{1}{2}\bigr)\Gamma\bigl(\frac{1}{2}\bigr)} {2}\frac{J_{j+k+\alpha}(\lambda t)}{\bigl(\frac{\lambda t}{2}\bigr)^{j+k+\alpha}}\\
&=t^{2(j+k+\alpha)}\mathcal{J}_{j+k+\alpha}(\lambda t)
\end{split}\]
where $\mathcal{J}_\mu(z)$ is defined by
$\mathcal{J}_\mu(z)=2^{\mu-1}\Gamma\bigl(\frac{1}{2}\bigr)\Gamma\bigl(\mu+\frac{1}{2}\bigr)z^{-\mu}J_\mu(z)$.
Upon close inspection (and a formal rearrangement of the two series which we
justify below), it is thus seen that
\begin{multline}\label{eqn.double-sum} 
\sum_{j=0}^\infty\sum_{k=0}^\infty d_j(4\cosh t)^{-j}\Bigl(\frac{\sinh t}{t}\Bigr)^{\alpha+j-\frac{1}{2}} a_k\bigl(t,\alpha+j-\tfrac{1}{2}\bigr) \int_0^t\cos(\lambda s)(t^2-s^2)^{j+k+\alpha-\frac{1}{2}}\,ds\\
\begin{split}
&=\sum_{j=0}^\infty\sum_{k=0}^\infty d_j(4\cosh t)^{-j}\Bigl(\frac{\sinh t}{t}\Bigr)^{\alpha+j-\frac{1}{2}}t^{2\alpha}t^{2(j+k)}\mathcal{J}_{\alpha+j+k}(\lambda t)\\
&=(\sinh t)^{\alpha-\frac{1}{2}}t^{\alpha+\frac{1}{2}}\sum_{j=0}^\infty\sum_{k=0}^\infty d_j(4\cosh t)^{-j}\Bigl(\frac{\sinh t}{t}\Bigr)^{j-\frac{1}{2}}t^{2(j+k)}\mathcal{J}_{\alpha+j+k}(\lambda t)\\
&=(\sinh t)^{\alpha-\frac{1}{2}}t^{\alpha+\frac{1}{2}}\sum_{m=0}^\infty a_m(t)t^{2m}\mathcal{J}_{\alpha+m}(\lambda t),\end{split}
\end{multline}
with
\[a_m(t)=\sum_{j=0}^m d_j(4\cosh t)^{-j}\Bigl(\frac{\sinh t}{t}\Bigr)^j a_{m-j}\bigl(t,\alpha+j-\tfrac{1}{2}\bigr).\] In other words,
\begin{equation}\label{eqn.expansion}
\varphi_\lambda^{(\alpha,\beta)}(t) =
\frac{2\Gamma(\alpha+1)}{\Gamma(\alpha+\tfrac{1}{2})\Gamma(\tfrac{1}{2})}\frac{t^{\alpha+\frac{1}{2}}}{(\sinh
t)^{\alpha+\frac{1}{2}}(\cosh t)^{\beta+\frac{1}{2}}} \sum_{m=0}^\infty
a_m(t)t^{2m}\mathcal{J}_{m+\alpha}(\lambda t),
\end{equation}
which is the Jacobi function analogue of
\cite[Formula~(2.4)]{Stanton-Tomas}, thus demonstrating that Jacobi
functions generally behave like Bessel functions close to $0$,
akin to spherical functions on rank one symmetric spaces.

In order to justify the formal rearrangement of the double series in
\eqref{eqn.double-sum}, we presently prove that it is absolutely convergent. To
this end we first notice that if $\vert\lambda t\vert \leq 1$, and $\lambda$
belongs to a set of the form $D_{\varepsilon,\gamma}$, we use the integral
formula \eqref{egn.Bessel.intrep} to write
\[\mathcal{J}_\mu(\lambda t)=\frac{1}{2}\int_{-1}^1e^{i(\lambda
t)s}(1-s^2)^{\mu-\frac{1}{2}}\,ds =
\frac{1}{2}\int_{-1}^1e^{i(\Re\lambda)ts-(\Im\lambda)ts}(1-s^2)^{\mu-\frac{1}{2}}\,ds, \] implying that
\begin{equation}\label{eqn.modifiedBessel.less1}
\vert\mathcal{J}_\mu(\lambda t)\vert\leq\frac{1}{2}e^{\varepsilon\vert\Re\lambda\vert t}\int_{-1}^1(1-s^2)^{\Re\mu-\frac{1}{2}}\,ds \leq
e^{\varepsilon\vert\lambda t\vert}\leq e^\varepsilon\text{ for }\vert\lambda
t\vert\leq 1.
\end{equation}
An even simpler bound is available if $\lambda$ belongs to a compact subset of
$\C\setminus(-i\N)$, so we shall presently ignore this possibility. An
analogous bound on $\vert\mathcal{J}_\mu(\lambda t)\vert$ for $\vert\lambda
t\vert\geq 1$ is obtained by observing that, still for $\vert\lambda t\vert\geq
1$ and $\lambda\in D_{\varepsilon,\gamma}$, one has
$\vert\mathcal{J}_\mu(\lambda t)\vert\leq\frac{1}{2}e^{\vert\Im
\lambda\vert t}\int_{-1}^1(1-s^2)^{\Re\mu-\frac{1}{2}}\,ds\leq e^{\gamma
R_0}$, which, while being a poor bound, is independent of $\mu$.
\medskip

A summand in \eqref{eqn.double-sum} is therefore bounded,
according to \eqref{eqn.ST.2.11} and
\eqref{eqn.modifiedBessel.less1}, by
\[\begin{split}
&c_\alpha\vert d_j\vert \vert 4\cosh t\vert^{-j}\biggl|\Bigl(\frac{\sinh t}{t}\Bigr)^{\alpha+j-\frac{1}{2}}a_k(t,\alpha+j-\tfrac{1}{2})\biggr| \vert t\vert^{2(j+k)}\\
\leq & c_\alpha\vert d_j\vert \vert 4\cosh t\vert^{-j}\Bigl(\frac{4\cosh t}{R_1}\Bigr)^{\Re\alpha+j-\frac{1}{2}}R_1^{-k}\\
=&c_\alpha\vert d_j\vert \vert 4\cosh t\vert^{\Re\alpha-\frac{1}{2}}\Bigl(\frac{t^2}{R_1}\Bigr)^{j+k} \end{split}\] since
$R_1>1$ and $\Re\alpha>-\frac{1}{2}$ by assumption. In addition,
$\frac{t^2}{R_1}\leq \frac{R_0^2}{R_1}<1$, and the  sequence $(d_j)$ has
already been shown to be bounded, so $\sum_j\sum_k\vert d_j\vert
(t^2R_1^{-1})^{j+k}$ converges. The sequence $(a_m)$  is therefore absolutely
summable and the rearrangement of the double series permissible.
\medskip

We can now prove the upper estimate \eqref{A.thm.asymptoticEXP.a.ests} for
$a_m(t)$: First note that $a_0(t)\equiv 1$ and
\[\begin{split}
\vert a_m(t)\vert &\leq c_\alpha\sum_{j=0}^m\vert d_j\vert \vert 4\cosh t\vert^{-j}\biggl|\Bigl(\frac{\sinh t}{t}\Bigr)^j a_{m-j}(t,\alpha+j-\tfrac{1}{2})\biggr|\\
&\leq c_\alpha\sum_{j=0}^m\vert d_j\vert \vert 4\cosh t\vert^{-j}\Bigl|\frac{4\cosh t}{R_1}\Bigr|^{\Re\alpha+j-\frac{1}{2}} R_1^{-(m-j)}\Bigl|\frac{\sinh t}{t}\Bigr|^{-(\Re\alpha-\frac{1}{2})}\\
&=c_\alpha'R_1^{-(\Re\alpha+m-\frac{1}{2})}\Bigl|\frac{t\cosh t}{\sinh t}\Bigr|^
{\Re\alpha-\frac{1}{2}}\sum_{j=0}^m\vert d_j\vert.
\end{split}\]
Since the series $\sum_j j^{-(\Re\alpha+\frac{1}{2})}$ is convergent,
one may estimate $\sum_{j=1}^m \vert d_j\vert$ with the convergent series
$\sum_{j=1}^\infty\vert d_j\vert$ (which certainly involves $\alpha$ and $\beta$ but not $R_1$, so it is immaterial exactly what this constant is).
As the function $t\mapsto\left|\frac{t\cosh
t}{\sinh t}\right|^{\Re\alpha-\frac{1}{2}}$ is bounded on $[0,R_0]$, we
have thus obtained the sought-after upper bound on $\vert a_m(t)\vert$.
\medskip

As for the error term analysis, it is insufficient to quote
\cite{Stanton-Tomas}, as we need to carry out the estimates for Bessel
functions with complex parameters and complex argument. The key idea, as
already used decisively in \cite{Szego}, it to employ integration by parts. It
can be shown by induction that for every nonzero integer $k\leq\Re\mu$,
there exists a degree $k$ polynomial $p_k$ with zeros in $t=\pm 1$ and a
constant $c$ such that
\[\int_{-1}^1e^{izs}(1-s^2)^{\mu-\frac{1}{2}}\,ds =
c\frac{(\mu-\frac{1}{2})^k}{(iz)^k}\int_{-1}^1
e^{izs}(1-s^2)^{\mu-\frac{1}{2}-k}p_k(s)\,ds.\] The point is that by choosing
$k$ large enough, we gain powers $\vert z\vert^{-k}=\vert\lambda t\vert^{-k}$,
which leads to a more favorable estimate of $\vert\mathcal{J}_\mu(\lambda t)\vert$
in the region where $\vert\lambda t\vert >1$. More precisely, the integral
formula for the Bessel function $J_\mu$ (valid whenever $\Re\mu>-\frac{1}{2}$)
\[\Gamma(\mu+\tfrac{1}{2})J_\mu(z)=\frac{1}{\sqrt{\pi}}\Bigl(\frac{z}{2}\Bigr)^\mu\int_{-1}^1e^{izt} (1-t^2)^{\mu-\frac{1}{2}}\,dt\]
implies that $\vert J_\mu(z)\vert\leq e^{-\vert\Im z\vert}\int_{-1}^1(1-s^2)^{\Re\mu-\frac{1}{2}}\,ds$. As
\[\int_{-1}^1(1-s^2)^{\mu-\frac{1}{2}-k}ds=\frac{\Gamma(\tfrac{1}{2})\Gamma(\mu+\tfrac{1}{2}-k)}
{\Gamma(\mu+1-k)}\] it thus follows from the definition of
$\mathcal{J}_\mu(z)$ that
\begin{equation}\label{eqn.est.Bessel.Szego}
\vert\mathcal{J}_\mu(z)\vert\lesssim
\frac{\vert\mu-\frac{1}{2}\vert^ke^{-\vert\Im z\vert}}{\vert
z\vert^k}\biggl|\frac{\Gamma(\tfrac{1}{2})\Gamma(\mu+\tfrac{1}{2}-k)}{\Gamma(\mu+1-k)}\biggr|
\end{equation}
for $k\in\N_0\cap[0,\Re\mu]$.
For the error term analysis, we thus take $E_{M+1}(\lambda t)$ to be
\[E_{M+1}(\lambda t) = \frac{2^{\alpha-\frac{1}{2}}\Gamma(\alpha+1)}{\Gamma(\alpha+\frac{1}{2})\Gamma(\frac{1}{2})} \frac{t^{\Re\alpha+\frac{1}{2}}}{\Delta'(t)}
\sum_{m=M+1}^\infty a_m(t)t^{2m}\mathcal{J}_{m+\alpha}(\lambda t).\] When
$\vert\lambda t\vert\leq 1$, the error term $E_{M+1}(\lambda t)$ is bounded by
\[\begin{split}
& c\biggl|\frac{t^{\Re\alpha+\frac{1}{2}}}{\Delta'(t)}\biggr|\sum_{m=M+1}^\infty\vert a_m(t)\vert t^{2m}\vert\mathcal{J}_{m+\alpha}(\lambda t)\vert\\
\lesssim & c\biggl|\frac{t^{\Re\alpha+\frac{1}{2}}}{\Delta'(t)}\biggr|\sum_{m=M+1}^\infty\Bigl|\frac{t\cosh t}{\sinh t}\Bigr|^{\Re\alpha-\frac{1}{2}} R_1^{-(\Re\alpha+m-\frac{1}{2})} t^{2m}\\
\lesssim & c'\underbrace{\frac{t^{2\Re\alpha}(\cosh t)^{\Re\alpha-\Re\beta-1}}{(\sinh t)^{2\Re\alpha}}}_\sharp \sum_{m=M+1}^\infty R_1^{-m}t^{2m}\\
\lesssim & c\Bigl(\frac{t^2}{R_1}\Bigr)^{M+1}\sum_{j=0}^\infty\Bigl(\frac{t^2}{R_1}\Bigr)^j \leq c\Bigl(\frac{t^2}{R_1}\Bigr)^{M+1}\sum_{j=0}^\infty \Bigl(\frac{R_0^2}{R_1}\Bigr)^j\\
\lesssim & c\Bigl(\frac{t^2}{R_1}\Bigr)^{M+1}< ct^{2(M+1)},
\end{split}\]
since the factor $(\sharp)$, as a function in $t$, is bounded on $[0,R_0]$. Observe that the constant $c$ is obtained in such a way that the decay in $t$
determined by $\alpha$ is being accounted for.

It makes sense to try and estimate $E_{M+1}(\lambda t)$ differently whenever
$\vert\lambda t\vert\geq 1$, as we might be able to introduce a certain amount
of decay. Indeed, this possibility was already observed and used in
\cite{Stanton-Tomas} and \cite{Schindler}. To this end we use the asymptotic
expansion of Bessel functions with complex parameter and complex argument, as
found in \cite[page~199, Formula~1]{Watson}, to write
\[\vert J_{m+\alpha}(\lambda t)\vert\lesssim \frac{1}{\vert\lambda t\vert^{\frac{1}{2}}}\Bigl|\cos\left(\lambda t-\frac{\pi(m+\alpha)}{2}-\frac{\pi}{4}\right)\Bigr|.\]
Since both $\lambda$ and $\alpha$ are allowed to be complex, we cannot simply
estimate the cosine with $1$ (as was done in the proof of
\cite[Theorem~2.1]{Stanton-Tomas}), and since the parameter $m+\alpha$ varies
with $m$, we could potentially end up with an upper bound on $\vert
J_{m+\alpha}(\lambda t)\vert$ that would get worse with increasing $m$. This is
not so, however: In the expression $\cos\bigl(\lambda
t-\frac{\pi(m+\alpha)}{2}\bigr)= \cos(\lambda
t)\cos(\frac{\pi(m+\alpha)}{2})+\sin(\lambda t)\sin(\frac{\pi(m+\alpha)}{2})$
we estimate $\cos(\frac{\pi(m+\alpha)}{2})=\cos(\frac{\pi
m}{2})\cos(\frac{\pi\alpha}{2})- \sin(\frac{\pi
m}{2})\sin(\frac{\pi\alpha}{2})$ by the constant $\sinh(\frac{\pi}{2}\Im\alpha)+\cosh(\frac{\pi}{2}\Im\alpha)$. Furthermore $\vert\cos(\lambda
t)\vert=\cosh((\Im\lambda)t)\leq \cosh(\gamma R_0)$ according to our
standing assumption on $\lambda$, so we may bound the modified Bessel function
in the first term in the series expression for $E_{M+1}(\lambda t)$ as follows:
\begin{equation}\label{eqn.improvedM}
\vert\mathcal{J}_{\alpha+M+1}(\lambda t)\vert \leq
c_{\alpha,M}\frac{2^{\Re\alpha+M}\sqrt{\pi}\vert\Gamma(\alpha+M+\frac{3}{2})\vert}{\vert\lambda
t\vert^{\Re\alpha+M+\frac{3}{2}}}.
\end{equation}

In order to effectively estimate the remaining terms in $E_{M+1}(\lambda t)$,
we use \eqref{eqn.est.Bessel.Szego} with $k=\lfloor\Re\alpha+M+2\rfloor$, the integer part of the number $\Re\alpha+M+2$,
yielding the slightly improved estimate
\begin{equation}\label{eqn.improved}
\vert\mathcal{J}_{m+\alpha}(\lambda t)\vert \lesssim \frac{\vert
m+\alpha-\frac{1}{2}\vert^{\lfloor\Re\alpha\rfloor+M+2}}{\vert\lambda
t\vert^{\lfloor\Re\alpha\rfloor+M+2}}
\biggl|\frac{\Gamma(\frac{1}{2})\Gamma(m-M+\alpha-\lfloor\Re\alpha\rfloor-\frac{3}{2})} {\Gamma(m-M+\alpha-\lfloor\Re\alpha\rfloor-1)}\biggr|
\end{equation}
for $m\geq M+2$; note here that
\[\biggl|\frac{\Gamma(m-M+\alpha-\lfloor\Re\alpha\rfloor-\frac{3}{2})}
{\Gamma(m-M+\alpha-\lfloor\Re\alpha\rfloor-1)}\biggr|\simeq
m^{-\frac{1}{2}}, \quad m\gg 1\] by the usual estimates for the Gamma
function. It thus follows that
\begin{multline*}
\biggl|\sum_{m=M+2}^\infty t^{2m}a_m(t)\mathcal{J}_{m+\alpha}(\lambda t)\biggr| \\
\lesssim R_1^{-(\Re\alpha-\frac{1}{2})}\sum_{m=M+2}^\infty
t^{2m}R_1^{-m}\bigl|m+\alpha-\frac{1}{2}\bigr|^{\lfloor\Re\alpha\rfloor+M+2}
(m-M)^{-\frac{1}{2}}\\
\lesssim R_1^{-(\Re\alpha-\frac{1}{2})}\sum_{m=M+2}^\infty\underbrace{\left(\frac{R_0^2}{R_1}\right)^m\bigl|m+\alpha-\frac{1}{2}\bigr|^{\lfloor\Re\alpha\rfloor+M+2}m^{-\frac{1}{2}}}_{b_m}
\end{multline*}
Since $\vert b_m\vert^{1/m}\simeq\frac{R_0^2}{R_1}\vert
m+\alpha-\frac{1}{2}\vert^{(\lfloor\Re\alpha\rfloor+M+2)/m}m^{-\frac{1}{2m}}$, with $\frac{R_0^2}{R_1}<1$, it is
clearly possible to find a (possibly large) integer $m_0$ such that
\[\forall m\geq m_0:\vert b_m\vert^{1/m}\lesssim \frac{1}{2}\left(1+\frac{R_0^2}{R_1}\right)<1.\]
We conclude that the series $\sum_{m=M+2}^\infty b_m$ is absolutely convergent -- this would \emph{not} follow had we instead used the weaker estimate \eqref{eqn.standard.Bessel-estimate}.

At long last we may now conclude, with the help of \eqref{eqn.improvedM}, that
\[\vert E_{M+1}(\lambda t)\vert\leq c_M t^{2(M+1)}\vert\lambda t\vert^{-(\Re\alpha+M+1)},\]
where $c_M$ is bounded by
\begin{multline*}
c\Biggl|\Bigl|\frac{2^{\alpha-\frac{1}{2}}\Gamma(\alpha+1)}{\Gamma(\alpha+\frac{1}{2})}\Bigr|
2^{\Re\alpha+M}R_1^{-(M+1)}\Gamma\bigl(\alpha+M+\tfrac{3}{2}\bigr)\\
+ \cdots c_\alpha 2^MR_0^{-2(M+1)}\sum_{m=M+2}^\infty \left(\frac{R_0^2}{R_1}\right)^m\bigl|
m+\alpha-\frac{1}{2}\bigr|^{\lfloor\Re\alpha\rfloor+M+2}m^{-\frac{1}{2}}\Biggr|<\infty
\end{multline*}
\end{proof}

As in \cite[Lemma~4]{Meaney-Prestini}, one may also estimate derivatives of
$\varphi_\lambda(t)$ and $E_{M+1}(\lambda t)$ with respect to the spectral
parameter $\lambda$. The somewhat stronger result reads as follows:

\begin{lemma}\label{lemma.derivatives}
\begin{enumerate}
\romannum \item For every nonnegative integer $n$, there exists a constant
$K_n\geq 0$ such that
\[\forall t\in\R^+, \lambda\in\C: \Bigl|\frac{d^n}{d\lambda^n}\varphi_\lambda(t)\Bigr|\leq K_n(1+t)^{n+1}e^{(\vert\Im z\vert - \Re\rho)t}.\]
In particular, $\bigl|\frac{d}{d\lambda}\varphi_\lambda(t)\bigr|\leq
K\vert\lambda\vert^{-2}e^{-\rho t}$ for $\vert\lambda t\vert<1$.

\item Assume $\vert\lambda t\vert< 1$. For every $M\geq 1$ there exists a
constant $c_M$ such that $\vert\frac{d}{d\lambda}E_{M+1}(\lambda t)\vert \leq
c_M t^{2(M+1)}\vert\lambda\vert^{-1}$.
\end{enumerate}
\end{lemma}
\begin{proof}
\begin{enumerate}\romannum
\item The stated estimate for $\frac{d^n}{d\lambda^n}\varphi_\lambda$ is proved
just as the real parameter-analogue in \cite[Lemma~14]{FJ}: It follows from the Laplace-type integral representation \eqref{eqn.Laplace-type} for $\varphi_\lambda$ that
\[\begin{split}
\frac{d^n}{d\lambda^n}\varphi_\lambda^{(\alpha,\beta)}(t) & = c_{\alpha,\beta}\int_0^1\int_0^\pi i^n(\ln\vert\cosh t+\sinh t\,re^{i\psi}\vert)^n \\
&\quad\times \vert\cosh t+\sinh t\,re^{i\psi}\vert^{i\lambda-\rho}(1-r^2)^{\alpha-\beta-1}r^{2\beta+1}(\sin\psi)^{2\beta}\,d\psi\,dr.
\end{split}\]
The stated estimate readily follows from a classical estimate of
$\varphi_\lambda(t)$, and a trivial estimate of the logarithm. More precisely,
according to \cite[Lemma~2.3]{Koornwinder-newproof},
$\vert\Gamma(\alpha+1)^{-1}\varphi_\lambda^{(\alpha,\beta)}(t)\vert\leq
(1+t)e^{(\vert\Im\lambda\vert-\Re\rho)t}$. Furthermore,
$\vert\cosh t+\sinh t\,re^{i\psi}\vert\leq e^t$ for $t\in[0,1]$,
$\psi\in[0,\pi]$, whence $(\ln\vert\cosh t+\sinh t\,re^{i\psi}\vert)^n\leq
t^n$.

\item The method of proof is the same as for
\cite[Lemma~4(ii)]{Meaney-Prestini}, but with
\eqref{eqn.standard.Bessel-estimate} giving the Bessel function estimates. The
proof goes as follows: We notice that
\[\begin{split}
\biggl|\frac{\partial}{\partial\lambda}\Bigl\{(\lambda
t)^{-(m+\alpha)}J_{m+\alpha}(\lambda t)\Bigr\}
\biggr| & \leq \frac{1}{t^{m+\Re\alpha}} \biggl\{\Bigl|\frac{(m+\alpha)J_{m+\alpha}(\lambda t)}{\lambda^{m+\alpha+1}}\Bigr| + \Bigl|\frac{tJ'_{m+\alpha}(\lambda t)}{\lambda^{m+\alpha}}\Bigr|\biggr\}\\
&\leq \frac{1}{t^{m+\Re\alpha}}\biggl\{\vert m+\alpha\vert\Bigl|\frac{J_{m+\alpha}(\lambda t)}{\lambda^{m+\alpha+1}}\Bigr|+ \Bigl|\frac{tJ_{m+\alpha+1}(\lambda t)}{\lambda^{m+\alpha}}\Bigr|\biggr\}
\end{split}\]
since $J_\mu'(z)=-J_{\mu+1}(z)+\frac{\mu}{z}J_\mu(z)$. By means of
\eqref{eqn.standard.Bessel-estimate}, the latter quantity is seen to be bounded
by
\begin{multline*}
\frac{1}{t^{m+\Re\alpha}}\biggl\{\vert\lambda\vert^{-1}\frac{t^{m+\Re\alpha}}{2^{m+\Re\alpha-1}}\frac{\vert m+\alpha\vert}{\vert
\Gamma(m+\alpha+1)\vert}
+\frac{t^2\vert\lambda\vert t^{m+\Re\alpha}}{2^{m+\Re\alpha+1}}\frac{e^{\vert\Im(\lambda t)}}{\vert\Gamma(m+\alpha+2)\vert}\Biggr\}\\
\lesssim \vert\lambda\vert^{-1}+\frac{\vert t\lambda\vert^2}{\vert\lambda\vert} \lesssim \vert\lambda\vert^{-1}.
\end{multline*}
The claim now follows by using the error estimates in the proof of Theorem
\ref{thm.asymptoticEXP} for the region $\vert\lambda t\vert\leq 1$.
\end{enumerate}
\end{proof}

\providecommand{\bysame}{\leavevmode\hbox to3em{\hrulefill}\thinspace}
\providecommand{\MR}{\relax\ifhmode\unskip\space\fi MR }
\providecommand{\MRhref}[2]{%
  \href{http://www.ams.org/mathscinet-getitem?mr=#1}{#2}
}
\providecommand{\href}[2]{#2}


\begin{thebibliography}{10}

\bibitem{Anker-Annals}
J.-P. Anker, \emph{{${\bf L}_p$} {F}ourier multipliers on {R}iemannian
  symmetric spaces of the noncompact type}, Ann. of Math. (2) \textbf{132}
  (1990), no.~3, 597--628.

\bibitem{Anker.duke}
\bysame, \emph{Sharp estimates for some functions of the {L}aplacian on
  noncompact symmetric spaces}, Duke Math. J. \textbf{65} (1992), no.~2,
  257--297.

\bibitem{Anker-Damek-Yacoub}
J.-P. Anker, E.~Damek, and C.~Yacoub, \emph{Spherical analysis on harmonic
  {$AN$} groups}, Ann. Scuola Norm. Sup. Pisa Cl. Sci. (4) \textbf{23} (1996),
  no.~4, 643--679 (1997).

\bibitem{Anker-Lohoue}
J.-P. Anker and N.~Lohou{\'e}, \emph{Multiplicateurs sur certains espaces
  sym\'etriques}, Amer. J. Math. \textbf{108} (1986), no.~6, 1303--1353.

\bibitem{Brandolini-Gigante}
L.~Brandolini and G.~Gigante, \emph{Equiconvergence theorems for
  {C}h\'ebli-{T}rim\`eche hypergroups}, Ann. Sc. Norm. Super. Pisa Cl. Sci. (5)
  \textbf{8} (2009), no.~2, 211--265.

\bibitem{ClercStein}
J.~L. Clerc and E.~M. Stein, \emph{{$L^{p}$}-multipliers for noncompact
  symmetric spaces}, Proc. Nat. Acad. Sci. U.S.A. \textbf{71} (1974),
  3911--3912.

\bibitem{CoifmanWeiss-book}
R.~R. Coifman and G.~Weiss, \emph{Analyse harmonique non-commutative sur
  certains espaces homog\`enes}, Lecture Notes in Mathematics, Vol. 242,
  Springer-Verlag, Berlin, 1971, {\'E}tude de certaines int{\'e}grales
  singuli{\`e}res.

\bibitem{CoifmanWeiss}
\bysame, \emph{Transference methods in analysis}, American Mathematical
  Society, Providence, R.I., 1976, Conference Board of the Mathematical
  Sciences Regional Conference Series in Mathematics, No. 31.

\bibitem{Cowling_Giulini_Meda1}
M.~Cowling, S.~Giulini, and S.~Meda, \emph{{$L^p$}-{$L^q$} estimates for
  functions of the {L}aplace-{B}eltrami operator on noncompact symmetric
  spaces. {I}}, Duke Math. J. \textbf{72} (1993), no.~1, 109--150.

\bibitem{Damek-Ricci}
E.~Damek and F.~Ricci, \emph{Harmonic analysis on solvable extensions of
  $h$-type groups}, J. geom. Anal. \textbf{2} (1992), no.~3, 213--248.

\bibitem{ErdelyiII}
A.~Erd{\'e}lyi, W.~Magnus, F.~Oberhettinger, and F.~G. Tricomi, \emph{Tables of
  integral transforms. {V}ol. {II}}, McGraw-Hill Book Company, Inc., New
  York-Toronto-London, 1954, Based, in part, on notes left by Harry Bateman.
  \MR{MR0065685 (16,468c)}

\bibitem{FJ}
M.~Flensted-Jensen, \emph{{P}aley--{W}iener type theorems for a differential
  operator connected with symmetric spaces}, Ark. Mat. \textbf{10} (1972),
  143--162.

\bibitem{Koornwinder-FJ}
M.~Flensted-Jensen and T.~Koornwinder, \emph{The convolution structure for
  {J}acobi function expansions}, Ark. Mat. \textbf{11} (1973), 245--262.

\bibitem{Gigante}
G.~Gigante, \emph{Transference for hypergroups}, Collect. Math. \textbf{52}
  (2001), no.~2, 127--155.

\bibitem{Giulini-Mauceri-Meda.Crelle}
S.~Giulini, G.~Mauceri, and S.~Meda, \emph{{$L^p$} multipliers on noncompact
  symmetric spaces}, J. reine angew. Math. \textbf{484} (1997), 151--175.

\bibitem{Helgason-PW}
S.~Helgason, \emph{An analogue of the {P}aley--{W}iener theorem for the
  {F}ourier transform of certain symmetric spaces}, Math. Annalen \textbf{165}
  (1966), 297--308.

\bibitem{Johansen-exp2}
T.~R. Johansen, \emph{On a class of non-integrable multipliers for the {J}acobi
  transform}, submitted (2010), 15 pages.

\bibitem{Johansen-disc}
\bysame, \emph{Almost everywhere convergence of the inverse {J}acobi transform
  and endpoint results for a disc multiplier}, Studia Math. \textbf{205}
  (2011), no.~2, 101--137.

\bibitem{Koornwinder-newproof}
T.~Koornwinder, \emph{A new proof of a {P}aley-{W}iener type theorem for the
  {J}acobi transform}, Ark. Mat. \textbf{13} (1975), 145--159.

\bibitem{Koornwinder-book}
T.~H. Koornwinder, \emph{Jacobi functions and analysis on noncompact semisimple
  {L}ie groups}, Special functions: group theoretical aspects and applications,
  Math. Appl., Reidel, Dordrecht, 1984, pp.~1--85.

\bibitem{Lohoue}
N.~Lohou{\'e}, \emph{Comparaison des champs de vecteurs et des puissances du
  laplacien sur une vari\'et\'e riemannienne \`a courbure non positive}, J.
  Funct. Anal. \textbf{61} (1985), no.~2, 164--201.

\bibitem{Meaney-Prestini}
C.~Meaney and E.~Prestini, \emph{Almost everywhere convergence of inverse
  spherical transforms on noncompact symmetric spaces}, J. Funct. Anal.
  \textbf{149} (1997), no.~2, 277--304.

\bibitem{Nilsson-pq}
A.~Nilsson, \emph{{$L^p$}-{$L^q$} multipliers on non-compact {R}iemannian
  symmetric spaces}, Math. Scand. \textbf{84} (1999), no.~2, 203--212.

\bibitem{Gestur_Henrik-SBT}
G.~{\'O}lafsson and H.~Schlichtkrull, \emph{The {S}egal-{B}argmann transform
  for the heat equation associated with root systems}, Adv. Math. \textbf{208}
  (2007), no.~1, 422--437.

\bibitem{Opdam}
E.~M. Opdam, \emph{Harmonic analysis for certain representations of graded
  {H}ecke algebras}, Acta Math. \textbf{175} (1995), no.~1, 75--121.

\bibitem{Schindler}
S.~Schindler, \emph{Some transplantation theorems for the generalized {M}ehler
  transform and related asymptotic expansions}, Trans. Amer. Math. Soc.
  \textbf{155} (1971), 257--291.

\bibitem{Stanton-Tomas}
R.~J. Stanton and P.~A. Tomas, \emph{Expansions for spherical functions on
  noncompact symmetric spaces}, Acta Math. \textbf{140} (1978), no.~3-4,
  251--276.

\bibitem{SteinWeiss}
E.M. Stein and G.~Weiss, \emph{Introduction to fourier analysis on euclidean
  spaces}, Princeton Mathematical Series, no.~32, Princeton University Press,
  Princeton, N.J., 1971.

\bibitem{Stromberg}
J.-O. Str{\"o}mberg, \emph{Weak type {$L\sp{1}$} estimates for maximal
  functions on noncompact symmetric spaces}, Ann. of Math. (2) \textbf{114}
  (1981), no.~1, 115--126.

\bibitem{Szego}
G.~Szeg{\"o}, \emph{{\"U}ber einige asymptotische {E}ntwicklungen der
  {L}egendreschen {F}unktionen}, Proc. London Math. Soc. (2) \textbf{36}
  (1932), 427--450.

\bibitem{Taylor.duke}
M.~E. Taylor, \emph{{$L^p$}-estimates on functions of the {L}aplace operator},
  Duke Math. J. \textbf{58} (1989), no.~3, 773--793.

\bibitem{Watson}
G.~N. Watson, \emph{A treatise on the theory of {B}essel functions}, second
  ed., Cambridge University Press, Cambridge, 1944.

\end{thebibliography}
\end{document}